\documentclass[11pt]{amsart}

\usepackage{amsmath}
\usepackage{amssymb}
\usepackage{graphicx}
\usepackage{amsmath} 
\usepackage{amssymb}
\usepackage{amsthm} 
\usepackage{caption}
\usepackage{cutwin}
\usepackage{graphicx}
\usepackage{subcaption}
\usepackage{enumerate}
\usepackage{float}
\usepackage{tabularx} 
\usepackage[dvipsnames, hyperref]{xcolor}
\usepackage{tikz}
\usetikzlibrary{arrows.meta,
                positioning,
                quotes
                }
\usepackage{wrapfig}
\usepackage[margin=1in,letterpaper]{geometry} 
\usepackage{cite}
\usepackage[final]{hyperref} 
\hypersetup{
    colorlinks=true,
    linkcolor=BlueViolet,
    citecolor=JungleGreen,  
    urlcolor=cyan
    }

\newtheorem{lemma}{Lemma}[section]
\newtheorem{corollary}{Corollary}

\newtheorem{theorem}{Theorem}[section]
\newtheorem{proposition}{Proposition}[section]
\newtheorem{remark}{Remark}[section]
\newtheorem{definition}{Definition}[section]

\def\R{\mathbb{R}}

\def\N{\mathbb{N}}

\def\x{\textbf{x}}

\def\ep{\epsilon}

\def\calW{\mathcal{W}}

\def\calY{\mathcal{Y}}
\def\calV{\mathcal{V}}

\def\ra{\rightarrow}

\newcommand\numberthis{\addtocounter{equation}{1}\tag{\theequation}}
\def\gamra{\stackrel{\Gamma}{\longrightarrow}}
\def\sqra{\stackrel{\square}{\longrightarrow}}

\def\weakstarra{\stackrel{\star}{\rightharpoonup}}
\def\GL{\operatorname{GL}}
\def\TV{\operatorname{TV}}

\numberwithin{equation}{section}

\title[]{Ginzburg--Landau Functionals in the Large-Graph Limit}

\author[E. J. Zhang]{Edith J. Zhang}
\author[J. Scott]{James Scott}
\author[Q. Du]{Qiang Du}
\author[M. A. Porter]{Mason A. Porter}

\address[E. J. Zhang]{Columbia University in the City of New York, NY, 10027, United States of America\footnote{Now affiliated with Department of Mathematics, University of California,
Los Angeles, CA, 90095, United States of America}}
\email{{\tt edith@math.ucla.edu}}

\address[J. M. Scott]{Department of Mathematics and Statistics, Auburn University, Auburn, Alabama, 36849, United States of America}
\email{{\tt james.m.scott@auburn.edu}}

\address[Q. Du]{Columbia University in the City of New York, NY, 10027, United States of America}
\email{{\tt qd2125@columbia.edu}}

\address[M. A. Porter]{Department of Mathematics, University of California,
Los Angeles, CA, 90095, United States of America}
\address[M. A. Porter]{Department of Sociology, University of California,
Los Angeles, CA, 90095, United States of America}
\address[M. A. Porter]{Santa Fe Institute, Santa Fe, NM, 87501, United States of America}
\email{{\tt mason@math.ucla.edu}}



\begin{document}

\begin{abstract}
Ginzburg--Landau (GL) functionals on graphs, which are relaxations of graph-cut functionals on graphs, have yielded a variety of insights in image segmentation and graph clustering. In this paper, we study large-graph limits of GL functionals by taking a functional-analytic view of graphs as nonlocal kernels. For a graph $W_n$ with $n$ nodes, the corresponding graph GL functional $\GL^{W_n}_\ep$ is an energy for functions on $W_n$. We minimize GL functionals on sequences of growing graphs that converge to functions called graphons. For such sequences of graphs, we show that the graph GL functional $\Gamma$-converges to a continuous and nonlocal functional that we call the \emph{graphon GL functional}. We also investigate the sharp-interface limits of the graph GL and graphon GL functionals, and we relate these limits to a nonlocal total-variation (TV) functional. We express the limiting GL functional in terms of Young measures and thereby obtain a probabilistic interpretation of the variational problem in the large-graph limit. Finally, to develop intuition about the graphon GL functional, we determine the GL minimizer for several example families of graphons.
\end{abstract}

\maketitle


\section{Introduction}

The study of large graphs has become increasingly common in network analysis, with applications {in sociology}, systems biology, communications, {epidemiology}, and many other areas \cite{sinews-graphons}. Because of their nice mathematical properties, graphons have become popular in both applications~\cite{medvedev2014b,parise2019graphon, ruiz2021graphon, gao2019graphon, vizuete2020graphon, wolfe2013nonparametric} and theoretical studies~\cite{borgs2016sparse, borgs2008convergent,braides2020cut,diaconis2007graph,lovasz2012large}. In particular, graphons are well-suited to variational problems on graphs because it is possible to continuously deform a graphon and thereby calculate variations of graphon functionals~\cite[Section 16.2]{lovasz2012large}. Graphons also connect graph functionals to nonlocal functionals~\cite{lou2010image, zhang2010wavelet}. For instance, one can view a graphon total-variation (TV) functional (which we will formulate in the present paper) as a nonlocal perimeter functional~\cite{elbouchairi2023nonlocal, caffarelli2009nonlocal, gilboa2009nonlocal}. Graphons {also have} been employed in the derivation of mathematically rigorous mean-field approximations of dynamical systems on graphs \cite{medvedev2014a,medvedev2014b}.

Optimization problems on graphs involve minimizing an energy functional that is defined on functions on graphs (i.e., functions that assign a value to each node of a graph). Optimization problems on large graphs are enormous combinatorial optimization problems, and they often are intractable computationally. One example is the minimum-cut (i.e., ``min-cut") problem, which arises in applications such as community detection~\cite{newman2018,van2013community} and image segmentation~\cite{rudin1992nonlinear, wu1993optimal, yi2012image, shi2000normalized,szlam2010total} and is related to the maximum-flow (i.e., ``max-flow") problem on networks~\cite{dantzig2003max,yuan2010study, ford1956maximal}.  

The classical min-cut problem entails partitioning the set of nodes of a graph into two subsets, $S$ and $S^c$, while minimizing the number of edges that one ``cuts" to separate $S$ and $S^c$. {The min-cut problem involves minimizing a \textit{graph-cut functional}}, which is equivalent to a graph TV functional~\cite{merkurjev2015global} {(see Remark \ref{rmk: graph-cut equivalence})}. 
As the size of available data increases, increasingly large graphs (with millions of nodes or more) occur in applications. Analyzing large graphs is computationally expensive. For an $n$-node graph (i.e., a graph of ``size" $n$), the min-cut problem involves optimizing over $2^n$ possible indicator functions, which each correspond to a possible partition $\{S, S^c\}$. This computational cost is a major obstacle in many applications. 

One {approach} to simplify computations in the min-cut problem is to use the graph Ginzburg--Landau (GL) functional. The GL functional is a relaxation of a graph-cut functional that is defined on $[-1,1]$-valued functions instead of on $\{-1,1\}$-valued functions. The $\Gamma$-convergence of the graph GL functional to the graph TV functional (equivalently, to the graph-cut functional), which was proved in \cite{van2012gamma}, justifies the use of the graph GL functional as a relaxation of the graph TV (i.e., graph-cut) functional. However, although the graph GL functional is easier to minimize than the graph-cut functional, the minimization process still relies on approximate algorithms~\cite{budd2019mbo, merkurjev2013mbo, van2014mean, bertozzi2016diffuse, shen2023preconditioned}. 

In the present paper, we further relax the graph-GL minimization problem to the continuum using a large-graph limit. That is, we evaluate the limiting minimization problem on functions {that map} $(0,1) \ra [-1,1]$ rather than on functions {that map} $\{1,\ldots, n\} \ra [-1,1]$. We use the idea of a \textit{graphon}~\cite{lovasz2012large}, which generalizes a graph's adjacency matrix (which is a linear operator that acts on vectors in $\R^n$) to a linear operator that acts on functions in $L^q(0,1)$. This viewpoint allows us to treat problems that involve large graphs as functional-analytic problems. 

Lovasz's original formulation of graphons~\cite{lovasz2012large} defined them as functions in $L^\infty((0,1)^2)$ that arise as limits of \textit{dense sequences of graphs}\footnote{Density is a property of a sequence of graphs, rather than a property of a graph itself, because the definition of density is based on the rate at which the number of edges increases in comparison to the number of nodes.} (i.e., sequences of graphs whose number of edges scales as $\Theta(n^2)$ in the number $n$ of nodes, which entails that the number of edges is bounded above by $c_1 n^2$ and {bounded} below by $c_2 n^2$ for some {positive constants $c_1$ and $c_2$}). Sequences of growing graphs in many applications and real-world situations are sparse~\cite{newman2018}, but applications of $L^\infty$ graphons typically involve only mean-field models \cite{bayraktar2023graphon, caines2021graphon}. To tackle this issue, researchers have defined \emph{$L^p$ graphons}, which arise as limits of sparse sequences of graphs~\cite{bollobas2007metrics,borgs2018LpII} (i.e., sequences of graphs whose number of edges increases as $o(n^2)$, which entails that the number of edges grows at a rate that is strictly less than $n^2$). One obtains $L^p$ graphons, which are associated with operators on $L^q((0,1)^2)$, by normalizing graphs by their edge density. Researchers have {also used traditional} $L^\infty$ graphons as limit objects of ``very sparse" sequences of graphs~\cite{bollobas2007metrics, bollobas2011sparse} (i.e., sequences of graphs with bounded degree or bounded mean degree as $n \ra \infty$). See Section \ref{sec: lp graphons} for further discussion of sparse graphons.

To consider the convergence of graph GL functionals in the graphon limit, we use a central tool in variational calculus that is known as $\Gamma$-convergence. {When} the domain of the underlying function space is compact, the $\Gamma$-convergence of a sequence $\{F_n\}$ of functionals $F_n$ to a limiting functional $F$ guarantees that the minimizers of $F_n$ converge to the minimizer of $F$~\cite{braides2006handbook}. 

The original GL theory is a physical model for phase transitions in superconductors~\cite{du1992analysis,cyrot1973ginzburg}. The GL functional, which is sometimes called the Allen--Cahn functional or the Modica--Mortola functional in some applications~\cite{allen1979microscopic,du2020phase,modica1977esempio}, is
\begin{equation}  \label{eq: classical GL}
    \GL_\ep (u) = \ep \int_\Omega \left|\nabla u^2\right| \,dx + \frac{1}{\ep} \int_\Omega\Phi(u(x)) \,dx \,, \quad u \in H^1(\Omega) \,,
\end{equation}
where $\Omega \in \R^d$ is a {bounded and connected set}, $\Phi$ is a double-well potential, and $H^1(\Omega) = W^{1,2}(\Omega) = \{f \in L^p(\Omega): \left(\|f\|_2^2 + \|f'\|_2^2\right)^{1/2} < \infty\}$ is a Sobolev space~\cite{evans2022partial}. As $\ep \ra 0$, the functional $\GL_\ep$ $\Gamma$-converges to the TV functional (i.e., perimeter function) \cite{bronsard1991motion}
\begin{equation}     \label{eq: classical TV}
    \TV(u) = \int_\Omega \left|\nabla u\right| \,dx \, .
\end{equation}

We take inspiration from van Gennip and Bertozzi~\cite{van2012gamma}, who defined the graph GL functional, which is a discrete version of $\GL_\ep$ for functions $u$ on graphs. See Section \ref{sec: graph functionals} for the definitions of the graph GL functional and other graph functionals. Van Gennip and Bertozzi proved $\Gamma$-convergence of the graph GL functional for a square-lattice graph. They derived both a large-graph (i.e., $n \ra \infty$) limit and a sharp-interface (i.e., $\ep \ra 0$) limit of this functional. As $n \ra \infty$, the square-lattice graph of grid size $1/n$ converges to the region $(0,1)^2 \subset \R^2$. 

A growing square-lattice graph is a specific sequence of growing graphs. One can view each square-lattice graph as a mesh approximation of the unit square. In the present paper, we generalize van Gennip and Bertozzi's results to general sequences of growing graphs that converge to a graphon as $n\ra \infty$. This convergence is guaranteed because the set of $L^p$ graphons is (under mild conditions) compact with respect to the cut metric \cite[Theorem 2.8]{borgs2019LpI}.

We refer to the limit of sequences of graph GL functionals as a \emph{graphon GL functional}. We show that the graphon GL functional $\Gamma$-converges to a nonlocal TV functional as $\ep \ra 0$.

\begin{figure}[H]
    \centering
        \begin{tikzpicture}[auto,
                       > = Stealth, 
           node distance = 22mm and 44mm,
              box/.style = {draw=gray, very thick,
                            minimum height=11mm, text width=22mm, 
                            align=center},
       every edge/.style = {draw, <->, very thick},
every edge quotes/.style = {font=\footnotesize, align=center, inner sep=1pt}]

    \node (n11) [box] {$\GL^{W_n}_\ep$~\eqref{eq: graph GL}};
    
    \node (n21) [box, xshift=4cm, yshift=-1cm, above=of n11] {$\TV^{W_n}$~\eqref{eq: graph TV}};
    
    \node (n31) [box, xshift=4cm, yshift=-5.7cm, above=of n21] {$\TV^W$~\eqref{eq: graphon TV}};
    
    \node (n32) [box, yshift=-1.2cm, below=of n21] {\large{$\GL^W_\ep$~\eqref{eq: graphon GL}}};
    
    \draw   [->] (n11) --node[midway,above left]{\large{$\stackrel{ \ep \,\ra\, 0}{\text{{\footnotesize limit (1)}}}$}}    (n21);
    
    \draw [->] (n21) --node[midway,above right]{\large{$\stackrel{n \,\ra\, \infty}{\text{{\footnotesize limit (2)}}}$}}     (n31);
    
    \draw [->] (n32) --node[midway,below right]{\large{$\stackrel{\ep \,\ra\, 0}{\text{{\footnotesize limit (4)}}}$ }}    (n31);
    
    \draw [->] (n11) --node[midway, below left]{\large{$\stackrel{n \,\ra\, \infty}{\text{{\footnotesize limit (3)}}}$}}  (n32);
    \end{tikzpicture}
    
    \caption{Four different $\Gamma$-convergences of the graph GL functional with different orderings of the $n \ra \infty$ and $\ep \ra 0$ limits. The arrows indicate $\Gamma$-convergence. For definitions of the functionals, see Section \ref{sec: notation}.
    } 
    \label{figure:convergences}
    \end{figure}
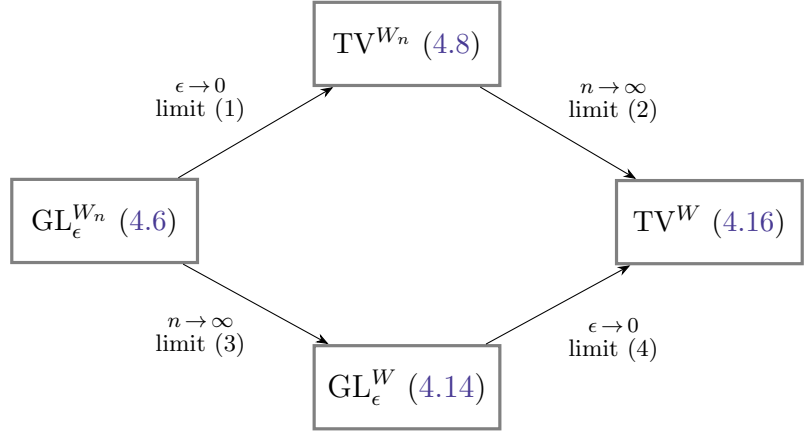

In Figure \ref{figure:convergences}, we illustrate two different sequential $\Gamma$-limits of the graph GL functional~\eqref{eq: graph GL}. One of the sequential limits is an $n \ra \infty$ and then $\ep \ra 0$ limit; the other is an $\ep \ra 0$ and then $n \ra \infty$ limit. In Sections \ref{sec: graph functionals} and \ref{sec: graphon functionals}, we define the associated limiting functionals. Limit (1) was proven in~\cite[Theorem 3.1]{van2012gamma}, and limit (2) follows from~\cite[Theorem 12]{braides2020cut}. To prove limit (3), we use an approach that is similar to the proof of \cite[Theorem 12]{braides2020cut}. We significantly generalize {\cite[Theorem 5.2]{van2012gamma}} in the sense that our graph limits {are valid for} general sequences of growing graphs, rather than only for a growing square-lattice graph. Limit (4) resembles the classical {Modica--Mortola} limit~\cite{modica1977esempio} (which states that $\GL_\ep \gamra {\TV}$), but it uses
graphon versions of $\GL$ and $\TV$ functionals. We prove limit (4) for {the graphons $W \in L^\infty((0,1)^2)$, which arise from dense sequences of graphs.} Limit (4) is not defined for more general $L^p$ graphons. 

Our results extend the study of graph GL {functionals} and {graph} TV functionals to graphon limits. We refer to these limiting functionals as the graphon GL {functionals} and {graphon} TV functionals, respectively. We show that the {graph GL-minimization and graph TV-minimization problems are consistent with the limiting graphon GL-minimization and graphon TV-minimization problems in the sense of $\Gamma$-convergence.} That is, the minimizers of the functionals converge to the minimizers of the limiting functionals. We also show that the classical limit $\GL_\ep \gamra {\TV}$ holds for graphons (i.e., $\GL^W_\ep \gamra {\TV^W}$). 

Our paper proceeds as follows. In Section \ref{sec: related}, we discuss a few key results from related work. In Section \ref{sec: graphons}, we review graphons and formalize the notion of a large-graph limit. In Section \ref{sec: notation}, we review the GL and TV functionals on graphs and graphons. We also collect some results about Young measures, discuss our function spaces, and define the relevant types of convergence. In Section \ref{sec: sequential ep then n}, we prove the sequential $\Gamma$-limits with $\ep \ra 0$ and then $n \ra \infty$. In Section \ref{sec: sequential n then ep part n}, we prove the sequential $\Gamma$-limits with $n \ra \infty$ and then $\ep \ra 0$. In Section \ref{sec: examples}, we {determine} GL minimizers for a few example families of graphons. Finally, in Section \ref{sec: conclusion}, we summarize our results and discuss future research directions.


\section{Related work}\label{sec: related}

Our work is inspired by van Gennip and Bertozzi~\cite{van2012gamma}, who examined the same four limits as in Figure \ref{figure:convergences} and also the simultaneous limit $\ep \ra 0$, $n \ra \infty$. They proved their $n \ra \infty$ limits specifically for
the square-lattice graph. We restate their version of limit (1), which holds for all graphs, and we extend their versions of the limits (2)--(4) to a nonlocal graph-limit scenario
in which the limiting GL and TV functionals are the graphon GL and TV functionals~\eqref{eq: graphon GL} and~\eqref{eq: graphon TV}, respectively. 
Limits (2)--(4) hold for general sequences of graphs. As we discuss in Section \ref{sec: graphon functionals}, the limiting graphon functionals are nonlocal.

We adapt ideas from Braides et al.~\cite{braides2020cut}, who used Young measures to show that the graph-cut functional (which we define later in \eqref{eq: graph-cut dirichlet form}) $\Gamma$-converges to the graphon-cut functional. We also use Young measures, which are generalizations of functions on graphs that assign 
a distribution of possible states (rather than a single fixed value) to each node of a graph. Young measures, which we define and discuss in Section \ref{sec: ym}, play a central role in our paper. Our methods and results are related more closely to the results of \cite{braides2020cut} than to those of~\cite{van2012gamma}. A key difference from \cite{braides2020cut} arises from the fact that
the domain of the cut functional {consists of} finite-range functions, whereas the domains of the graph and graphon GL functionals {consist of} continuous-range functions. 

{Garc{\'\i}a Trillos and Slep{\v{c}}ev}~\cite{trillos2013gamma} also studied a local limit and, analogously to limit (2), obtained $\Gamma$-convergence of the perimeter functional for the $n \ra \infty$ limit of point clouds. {The convergence of their graph functions, which are defined on point clouds, occurs in a metric space (which is called the $TL^p$ space) that is characterized by optimal-transport maps. The $TL^p$ space has also been employed to obtain results for other} large-graph limits of point clouds (see, e.g., \cite{trillos2015variational, garcia2016continuum, garcia2020maximum, akash2022wasserstein}). 
A key difference between our paper and research
on $TL^p$ limits is that we do not require the limiting functional to be a local quantity, such as Euclidean perimeter or the usual Dirichlet energy. Instead, our limiting functionals are nonlocal limits with interaction potentials that are given by graphons. 
These functionals are ``relaxations" in the sense that one can recover local functionals from them as special cases for certain graphons that have singularities along the line $y = x$.

The original theory of graphons~\cite{lovasz2012large} {is concerned with} dense sequences of graphs, and Braides et al.~\cite{braides2020cut} also required the graphs in such sequences to be dense. However, most real-world graphs are sparse~\cite{newman2018, bollobas2007metrics}, so this density requirement is a major limitation of much research on graphons. Thankfully, the theory of graphons has been extended to {sparse sequences of graphs}~\cite{borgs2019LpI, borgs2018LpII}, and our analysis allows sparse sequences of graphs that converge to $L^p$ graphons. See Section \ref{sec: lp graphons} for more detail.

One recovers different types of TV functionals when taking the $\ep \ra 0$ limit of the classical GL, graph GL, and graphon GL functionals. In particular, we obtain a nonlocal, continuous TV functional in the $\ep \ra 0$ limit of the graphon GL {functional}. 
Nonlocal TV functionals have been useful in a variety of applications, especially in image processing~\cite{caffarelli2009nonlocal, elbouchairi2023nonlocal, gilboa2009nonlocal, liu2014new, lou2010image, arias2012nonlocal, aujol2015fundamentals, el2014mean, hafiene2018nonlocal}. They are also of theoretical interest because they generalize the notion of perimeter {from objects (e.g., ones in $\R^2$) with geometric regularity, on which one can compute gradients of functions, to less regular objects (such as objects in metric spaces)\cite{mazon2019nonlocal}.} 
The fractional GL functional is a well-studied example of a nonlocal GL functional~\cite{savin2012gamma, tarasov2005fractional, pu2013well}, but thus far it has not been connected to graph theory. 
A graphon GL {functional} is a fractional GL {functional} when a graphon is of the form $W(x,y) = \frac{1}{|x - y|^{1 + 2s}}$, with $s \in (0,1)$.


\section{Graphons, norms, and convergence}\label{sec: graphons} 

A ``graphon", which is a portmanteau of ``graph" and ``function", is a bounded, measurable, and symmetric function $W: \Omega^2 \ra \R$, where $\Omega \subset \R^d$ is a {connected and bounded} domain. The set of graphons is ${\mathcal{W}} = \{ W: \Omega^2 \ra \R\}$. Throughout this paper, we take $\Omega = (0,1)$. The closed interval $[0,1]$ is typically used in the graphon literature, whereas the  open interval $(0,1)$ is typically used in functional analysis {{to} avoid complications that relate to
isolated points. In the present paper, it makes no difference because {we always integrate $W$.}} We use the open interval to be consistent with the conventions of functional analysis, which provides the main technical machinery in our paper. More generally, one can choose $\Omega$ to be any domain in $\R^d$. 

We consider graphs that are weighted, undirected, and simple (i.e., there are no self-edges or multi-edges). Let $W_n$ denote a graph with the node set $[n] = \{1, \ldots, n\}$. {The graph $W_n$} has an associated adjacency matrix $A^{(n)}$ with entries 
$A^{(n)}_{ij}$. We {also} associate the graph $W_n$ with a function $W_n$ that takes the constant value $A^{(n)}_{ij}$ on the product interval $I_i \times I_j$, where 
\begin{equation}\label{eq: def interval}
    I_i = [(i - 1)/n, \; i/n) \text{ for } i \in \{2,\ldots, n \} \text{ and } I_1 = (0, 1/n)\,.     
\end{equation}
The relationship between the adjacency matrix and the graphon is thus 
\begin{equation}   \label{eq: step graphon} 
    W_n(x,y) = \left\{ A^{(n)}_{ij} \;\; \text{ for } (x,y) \in I_i \times I_j\,, \,\,  i,j \in {[n]}  \right\}  \,.
\end{equation}
We have thereby associated a graph with a step function by associating the $i$th node with the interval $I_i$ and associating each edge $(i,j)$ with the product interval $I_i \times I_j$. In this way, one can identify each graph $W_n$ with a graphon, which we will also denote by $W_n$ (see Remark \ref{remark: identify W_n}). We refer to a graphon that corresponds to a graph as a \textit{step graphon}.  

In Figure~\ref{fig: example graphon}, we show an example graph and its corresponding adjacency matrix and step graphon. By convention, the axes of the graphon begin at the upper left.

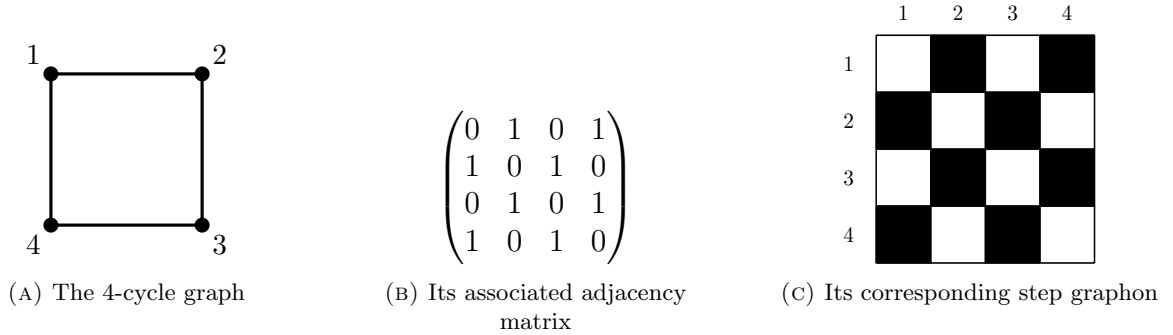
\begin{figure}[H]
\captionsetup[subfigure]{font=footnotesize}
\centering
    \subcaptionbox{The 4-cycle {graph} }[.32\textwidth]{%
        \begin{tikzpicture}
        \node at (0,0)[circle,fill,inner sep=2pt]{};
        \node at (0,2)[circle,fill,inner sep=2pt]{};
        \node at (2,0)[circle,fill,inner sep=2pt]{};
        \node at (2,2)[circle,fill,inner sep=2pt]{};

        \draw (0,2) node[anchor=south east] {$1$};
        \draw (2,2) node[anchor=south west] {$2$};
        \draw (2,0) node[anchor=north west] {$3$};
        \draw (0,0) node[anchor=north east] {$4$};
        \draw[-,very thick] (0,0)--(0,2);
        \draw[-,very thick] (0,0)--(2,0);
        \draw[-,very thick] (0,2)--(2,2);
        \draw[-,very thick] (2,0)--(2,2);
    \end{tikzpicture}
    }
    \subcaptionbox{Its associated adjacency \\ {matrix}}[.32\textwidth]{
    {\large $\begin{pmatrix}
        0&1&0&1
        \\
        1&0&1&0
        \\
        0&1&0&1
        \\
        1&0&1&0
    \end{pmatrix}$}
    }
    \hspace{.2cm}
    \subcaptionbox{Its corresponding step {graphon}}[.32\textwidth]{
    \resizebox{100pt}{100pt}{%
        \begin{tikzpicture}
            \draw[help lines, color=black, solid, thick, step = 1] (0,0) grid (4,4);
            
            \foreach \x in {1,..., 4}
                \node at (\x-.5, 4+.4) {\x} node at (0-.5, 4-\x+.5) {\x};
                
            \draw[color=black, fill = black] (1,3) rectangle (2,4); 
            \draw[color=black, fill = black] (0,2) rectangle (1,3); 
            \draw[color=black, fill = black] (3,3) rectangle (4,4); 
            \draw[color=black, fill = black] (2,2) rectangle (3,3); 
            \draw[color=black, fill = black] (1,1) rectangle (2,2); 
            \draw[color=black, fill = black] (3,1) rectangle (4,2); 
            \draw[color=black, fill = black] (0,0) rectangle (1,1); 
            \draw[color=black, fill = black] (2,0) rectangle (3,1);  
\end{tikzpicture}
}}
\caption{An example of a graph, its associated adjacency matrix, and its corresponding step graphon.}
\label{fig: example graphon}
\end{figure}

One can view any graphon $W: (0,1)^2 \ra \R$ as a large-graph limit by thinking of $(0,1)$ as a set with infinitely many nodes and taking $W(x,y)$ to be the weight of the edge between nodes $x \in (0,1)$ and $y \in (0,1)$. One can use graphons to represent both (1) families of graphs (by using a graphon as a probability distribution to generate graphs) and (2) limits of sequences of growing graphs. We employ the latter interpretation of graphons.


\subsection{Cut norm and cut convergence}

The {graphon-cut} norm (which is also known as the ``cut norm" and is closely related to the graph-cut functional) is the choice of topology for the space of graphons. {Accordingly,} both graphs and graphons converge to graphons with respect to cut norm. We introduce the graph-cut functional (which is sometimes called simply a ``cut functional") before introducing the cut norm. 

\begin{definition}[{graph cut}]
    For a partition $\{S, S^c\}$ of the nodes $[n]$ of a graph with adjacency matrix $A^{(n)}$, the {graph cut} is the functional
    \begin{equation}
        \mathrm{Cut}(S, S^c) = \sum_{i \in S, \, j \in S^c} A^{(n)}_{ij}  \,. 
    \end{equation}
    Equivalently, we express the graph-cut functional in terms of a graph function $u$ that takes values in the set $\{-1,1\}$ {by} using the expression 
    \begin{equation}\label{eq: graph-cut dirichlet form}
        \mathrm{Cut}(u) = {\frac{1}{8}} \sum_{i, j = 1}^n A^{(n)}_{ij} |u_i - u_j|^2 \quad \text{ with }\,\, u: [n] \ra \{-1,1\} \,.
    \end{equation}
\end{definition}

We define the cut norm $\|\cdot\|_\square$ on the space of graphons. The cut norm is closely related to the graph-cut functional, and it induces a metric such that $\mathcal{W}$ is a compact metric space~\cite{lovasz2007szemeredi}. Therefore, any bounded sequence $\{W_n\}_{n\in \N}$ of graphons has a subsequence {$\{W_{n'}\}$} such that $\|W_{n'} - W\|_\square \ra 0$.

\begin{definition}
    The \emph{cut norm} of a graphon $W$ is
\begin{equation} 
    \|W\|_\square = \sup_{S\subseteq (0,1)} \int_{S \times S^c} W(x,y)\,dx\,dy \,. 
\label{eq: cut norm} 
\end{equation}
\end{definition}

Definition \ref{eq: cut norm} highlights the similarity between the cut norm and the graph-cut functional. When a graphon is a step graphon $W_n$ (see equation~\eqref{eq: step graphon}), the cut norm~\eqref{eq: cut norm} becomes the finite sum
\begin{equation*}
	\|W_n\|_\square = \frac{1}{n^2} \sup_{S \subseteq [n]} \sum_{i\in S, \, j\in S^c} A^{(n)}_{ij} \,, 
\end{equation*}
which is the maximum {graph cut}, normalized by $1/n^2$, over all partitions $\{S, S^c\}$ of the nodes of $W_n$. There are other equivalent definitions of the cut norm \eqref{eq: cut norm}~\cite{janson2010graphons}. A particularly useful one for the present paper is
\begin{equation}
    \| W \|_\square = \sup_{\phi, \psi \in L^\infty((0,1);[-1,1])} \int_0^1 \int_0^1 W(x,y) {\phi}(x){\psi}(y)\,dx\,dy \,,
\label{eq: cut norm 1}
\end{equation}
where 
\begin{equation}
	L^\infty((0,1);[-1,1]) = \{f: (0,1) \ra [-1,1]\}\,.
\end{equation}
By normalizing $\phi \in L^\infty((0,1))$ to $\phi/\|\phi\|_\infty \in L^\infty((0,1);[-1,1])$, we obtain the equivalent definition
\begin{equation}
    \|W\|_\square = 
        {\sup_{\phi, \psi \in L^\infty((0,1))}}
        \frac{1}{\|\phi\|_\infty} \frac{1}{\|\psi\|_\infty} 
 \int_0^1 \int_0^1 \phi(x) \psi(y)W(x,y)\,dx\,dy \,.
\label{eq: cut norm 2}
\end{equation}

\begin{remark}\label{remark: identify W_n}
    One can identify any weighted graph with a step function $(0,1)^2 \ra \R$, and vice versa. We use the notation $W_n$ for both objects. 
    In concert with the fact that step functions are dense in $L^p$, one can approximate any graphon arbitrarily closely in cut norm by a graph. See~\cite[Section 3.3]{borgs2012convergent}, \cite[Remark 4.6]{janson2010graphons}, and~\cite{lovasz2007szemeredi}.
\end{remark}

\begin{remark}
    We use the notation ``$\sqra$" to denote convergence in cut norm. Accordingly, $W_n \sqra W$ {signifies}
    that $\|W_n - W\|_\square \ra 0$. 
\end{remark}

\begin{remark}
    The cut norm is equivalent to the operator norm of the kernel operator that is induced by the graphon $T_W(f) = \int_\Omega W(x,y) f(y) \, dy$, which is a linear operator $T_W: L^\infty((0,1)) \ra L^1(0,1)$. 
    (The norms $\|\cdot\|_A$ and $\|\cdot\|_B$ 
    are equivalent when
    $c_1 \|\cdot\|_A \leq \|\cdot \|_B \leq c_2 \|\cdot\|_A$ for some constants {$c_1$ and $c_2$}.)Additionally, people sometimes use the terms ``graphon" and ``kernel" interchangeably~\cite{janson2010graphons}.
\end{remark}

\begin{remark}
    {It is known that $\|W\|_\square \leq \|W\|_1$, where $\|\cdot\|_1$ is the $L^1(\Omega^2,\R)$ norm, for any graphon $W$~\cite{janson2010graphons}.} Additionally, for any step graphon with $n$ steps, $\|W_n\|_1 \leq \sqrt{2n} \|W_n\|_\square$~\cite[Equation 8.15]{lovasz2012large}.     Consequently, for each step graphon, the $L^1$ norm and cut norm are equivalent for finite $n$. However, this is not true when $n \ra \infty$. 
\end{remark}


\subsection{\texorpdfstring{$L^p$ graphons}{Lp graphons}}\label{sec: lp graphons}

The classical graphons that we discussed in Section \ref{sec: graphons} are known as ``$L^\infty$ graphons". They are functions in the space
\begin{equation}
	L^\infty((0,1)^2) = \{ W\;:\; { \|W\|_\infty := \text{ess sup}}_{x,y\in (0,1)^2} |W(x,y)| < \infty\} \,,
\end{equation}
where the essential supremum ess sup equals the supremum up to a measure-0 set.

The set of $L^\infty$ graphons~\cite{lovasz2012large} was built for dense sequences of graphs, 
which are sequences of graphs whose number of edges scales with the number $n$ of nodes as {$\Theta(n^2)$ (i.e., it scales asymptotically in proportion to $n^2$). 
Any sequence of graphs with $o(n^2)$ (i.e., asymptotically strictly less than $cn^2$ for some constant $c > 0$) edges converges to the zero graphon $W \equiv 0$ because the edge set is a set of measure $0$ in the $n \ra \infty$ limit. One can see this because the Riemann sum
\begin{equation*}
    \int_0^1 \int_0^1 W_n(x,y) \, dx\,dy = \frac{1}{n^2} \sum_{i,j = 1}^n A^{(n)}_{ij}
\end{equation*}
is nonzero in the $n\ra \infty$ limit only if the number of nonzero terms in the adjacency matrix $A^{(n)} = \{A^{(n)}_{ij}\}_{i,j = 1}^n$ is {$\Theta(n^2)$.} 

It is common to use $\calW$ to denote the set of $L^\infty$ graphons and to use $\calW_0 \subset \calW$ to denote the set of graphons that take values in $[0,1]$. However, one can identify any $W \in \calW$ with its normalized version $W/\|W\|_\infty \in \calW_0$. Therefore, we refer to both $\calW$ and $\calW_0$ as $L^\infty$ graphons to contrast them with $L^p$ graphons, which we will define shortly.

Although most real-world graphs are sparse, they often also have some nodes with degrees that one expects to grow linearly with $n$~\cite{newman2018}. For example, perhaps the total number of edges of a graph scales as $\Omega(n)$ (and perhaps the graph has a heavy-tailed degree distribution). If we use the relationship \eqref{eq: step graphon} to define graphons that correspond to sparse graphs, we obtain graphons that are nonzero only on sets of measure $0$. (This occurs because the number $I_i \times I_j$  of grid points is of order $n$, whereas anything with nonzero area must have order $n^2$.) To extend the theory of graphons to sparse sequences of graphs, Borgs et al.~\cite{borgs2019LpI, borgs2018LpII} introduced $L^p$ graphons (see \cite[Definition 2.7]{borgs2019LpI}), which allow graphon theory to encompass a much wider variety of sparse-graph sequences.

The set of $L^p$ graphons extends the set of $L^\infty$ graphons to the space 
\begin{equation}
	L^p((0,1)^2) = \left\{W\; : \; \|W\|_p:=\left(\int_0^1\int_0^1 |W(x,y)|^p \,dx\,dy\right)^{\frac{1}{p}} < \infty \right\} 
\end{equation}
for $p \geq 1$.  See \cite[Theorem 2.8]{borgs2019LpI} for a characterization of the sequences of graphons that converge to an $L^p$ graphon. In the present paper, we view graphons as functions in $L^p((0,1)^2)$ for $p \in [1,\infty]$, 
{so} we assume any necessary properties on step graphons that allow them to be $L^p$ functions.
 
When $\Omega$ is a bounded domain, $L^p(\Omega) \subset L^q(\Omega)$ for $p > q \geq 1$. Therefore, the set of $L^\infty$ graphons, which includes all dense graphs, is contained in the set of $L^p$ graphons. Similarly, every $L^p$ graphon is also an $L^1$ graphon. In the proofs of limits (2) and (3) (see Sections \ref{sec: limit (2)} and \ref{sec: limit (3)}), we {assume that}
$W_n$ and $W$ are in {$L^1((0,1)^2)$.} 
In the proof of limit (4) (see Section \ref{sec: limit (4)}{)}, we assume that $W_n$ and $W$ are in $L^\infty((0,1)^2)$.


\section{Functions and functionals on graphs and graphons}\label{sec: notation}

Because graphons are functions, it is natural to study them using ideas from functional analysis. Relevant notions include convergence, compactness, functionals on graphs and graphons, and $\Gamma$-convergence of those functionals. It is also appropriate to analyze the function spaces on which the functionals act.

We consider GL functionals and TV functionals, which act on ``graph functions" (i.e., functions on graphs) and ``graphon functions" (i.e., functions on graphons), respectively. Both the GL and TV functionals have classical, graph, and graphon versions. We defined the classical GL functional in~\eqref{eq: classical GL} and the classical TV functional in~\eqref{eq: classical TV}. The graph GL and TV {functionals} are discrete and were defined {in~\cite{bertozzi2012diffuse} and~\cite{szlam2010total}, respectively.} We discuss them in {Section \ref{sec: graph functionals}.}
In Section \ref{sec: graphon functionals}, we define graph GL and TV functionals, which are continuous and involve Young measures.

The classical GL {functional \eqref{eq: classical GL}} is defined {on} $W^{1,2}((0,1))$, which is the Sobolev space of functions on $(0,1)$ that are $L^2$ with first derivatives that are also $L^2$ \cite[Chapter 5]{evans2022partial}. 
The classical {TV functional \eqref{eq: classical TV}} is defined on $W^{1,1}(0,1)$. We extend the domain of both functionals to $L^\infty((0,1))$ by imposing the value $+\infty$ whenever they otherwise would be undefined.


\subsection{Function spaces}\label{four-one}

We consider two types of function spaces: (1) spaces in which graphons live (i.e., function spaces of graphons) and (2) spaces of functions that are defined on graphons (i.e., function spaces of functions on graphons). As we discussed in Section \ref{sec: graphons}, graphons are always $L^p$ functions {for $p\in [1,\infty]$}. The functions on graphons that we consider are always {functions in $L^\infty((0,1))$}. 

The space of functions on $n$-node graphs is 
\begin{equation}
	\mathcal{V}^n = \{\hat{u}: [n] \ra \R\} \,.
\end{equation}	
Each function $\hat{u} \in \mathcal{V}^n$ has an associated step function $u: (0,1)\ra \R$. Let $u(x) = \hat{u}(i)$ for $x\in I_i$ for the  intervals $I_i$ that we defined in \eqref{eq: def interval}. 
This identification embeds $\mathcal{V}^n$ into the space $L^\infty((0,1))$ of bounded functions. Henceforth, we identify both $\hat{u}$ and $u$ as $u$. We also consider the subset 
\begin{equation}
	\mathcal{V}^n_b = \{u: [n] \ra \R\,, u_i \in \{\pm 1\} \,\,\text{for all}\,\, i\} 
\end{equation}	
of $\mathcal{V}^n$ that consists of binary graph functions.


\subsection{\texorpdfstring{$\Gamma$-convergence}{Gamma-convergence}}\label{sec: gamma convergence}

The notion of $\Gamma$-convergence of functionals is useful in optimization and the calculus of variations~\cite{braides2006handbook, van2012gamma}. In concert with a certain compactness property {(which we will specify shortly)}, the $\Gamma$-convergence of a sequence of functionals guarantees the convergence of corresponding minimizers {in} the sequence of functionals.  

\begin{definition} \label{def: gamma convergence}
    Let $X$ be a metric space, and consider a sequence of functionals $F_n: X \ra \R \cup \{\pm \infty\}$ for $n\in \{1,2,\ldots\}$. We say that $F_n$ \emph{$\Gamma$-converges} to $F: X \ra \R \cup \{\pm \infty\}$, which we denote by $F_n \gamra F$, with respect to $u_n \ra u$ {if 
    \begin{enumerate}[(i)]
        \item{{$\liminf_{n\ra\infty} F_n(u_n) \geq F(u)$ for every sequence $\{u_n\}$ such that $u_n \ra u$;}}
        \item{there exists a sequence $\{u_n\}_{n\in \N}$ such that $\limsup_{n\ra \infty} F_n(u_n) \leq F(u)$.} 
    \end{enumerate}}
\end{definition}

\begin{remark}\label{remark: gamma convergence norm}
    The definition of $\Gamma$-convergence of functionals requires a choice of norm for the convergence $u_n \ra u$. Therefore, $\Gamma$-convergence occurs \emph{with respect to} the convergence $u_n \ra u$. 
\end{remark}

If it is also true that any sequence $\{u_n\}_{n = 1}^\infty$ for which $\{F_n(u_n)\}_{n=1}^\infty$ is uniformly bounded has a convergent subsequence $\{u_{n_k}\}$, then the corresponding minimizers of $F_n$ converge to the minimizer(s) of $F$. This criterion, which we call the ``compactness property", is sometimes called the ``equicoerciveness property"~\cite{van2012gamma}.

It is useful to be purposeful when choosing the metric under which $u_n$ converges to $u$. When the convergence metric is stronger (where convergence in a stronger metric implies convergence in a weaker metric), {it} is easier to prove the lim inf inequality (i.e., the inequality in (i)) for $\Gamma$-convergence but harder to construct a {sub}sequence for the lim sup inequality (the inequality in (ii)) and for compactness. Conversely, when the convergence metric is weaker, it is harder to prove the lim inf inequality but easier to construct a subsequence for the lim sup inequality. In the present paper, we prove $\Gamma$-convergence results with respect to narrow convergence of Young measures (see Definition \ref{definition narrow convergence}). 


\subsection{Graph functions and functionals}\label{sec: graph functionals}

Recall the classical GL functional~\eqref{eq: classical GL}, which is defined on functions $u: (0,1) \ra \R$ by 
\begin{equation}
	\GL_\ep(u) = \ep \int_0^1 |\nabla u|^2 \,dx + \frac{1}{\ep} \int_0^1 \Phi(u(x)) \,dx \,,
\end{equation}	
where $\Phi$ is a double-well potential that has zeros at $s = \pm 1$. The double-well potential $\Phi$ can take a general form {(see~\cite[assumptions ({\it W}$_1$)--({\it W}$_4$)]{van2012gamma}),} but we use the standard choice 
\begin{equation}
	\Phi(s) = (s^2 - 1)^2 \,.
\end{equation}	

The graph version $\GL^{W_n}_\ep$ of the GL functional is analogous to $\GL_\ep$. It acts on $u \in \mathcal{V}^n$ {instead} of on $u\in L^\infty((0,1))$. We replace the gradient term $|\nabla u|^2$ by a finite-difference term that is weighted by the adjacency matrix, and we replace the double-well integral by a finite sum. (See {\cite[Section 2.2]{van2012gamma}} for further discussion of graph analogues of calculus operators.) {For $u \in \calV^n$,} we thus obtain 
\begin{align} 
    \GL^{W_n}_\ep (u) &= \frac{1}{n^2}\sum_{i,j = 1}^n A^{(n)}_{ij} |u_i - u_j|^2 + \frac{1}{\ep n} \sum_{i = 1}^n \Phi(u_i) \label{line_one} \\
   		 &= \int_0^1\int_0^1 W_n(x,y)|u(x) - u(y)|^2 \, dx \, dy + \frac{1}{\ep} \int_0^1 \Phi(u(x))\,dx \,, \label{eq: graph GL}
\end{align}
where~\eqref{line_one} uses the adjacency matrix and \eqref{eq: graph GL} {uses the definition~\eqref{eq: step graphon} of} the step graphon $W_n$. Similarly, the graph TV functional replaces the term $|\nabla u|$ in~\eqref{eq: classical TV} with a finite difference. It is finite only for binary functions. The graph TV functional is
\begin{align} 
    \TV^{W_n}(u) &= \begin{cases}
        \sum_{i,j=1}^n A^{(n)}_{ij} |u_i - u_j| \;\; &\text{ if } \,\, u \in \mathcal{V}^n_b 
        \\
        +\infty \;\; &\text{ if } \,\, u \in \mathcal{V}^n \setminus \mathcal{V}^n_b 
    \end{cases} \label{eq: graph TV discrete}
    \\
    &= \begin{cases}
    			\int_0^1 \int_0^1 W_n(x,y) |u(x)-u(y)| \,dx\,dy \;\; &\text{ if } \,\, u \in \mathcal{V}^n_b   \\
			    + \infty \;\; &\text{ if } \,\, u \in \mathcal{V}^n \setminus \mathcal{V}^n_b \,. 
\end{cases} \label{eq: graph TV}
\end{align}

The Dirichlet energy 
\begin{equation}\label{eq: classical Dirichlet}
    D(u) = \int_0^1 |\nabla u(x)|^2 \,dx 
\end{equation}
{has a similar form {as}} the graph-cut functional, but it acts on $W^{1,2}(0,1)$ functions. The Sobolev embedding of $W^{1,2}(0,1)$ into $L^\infty(0,1)$ allows us to define the Dirichlet energy $D$ on all $L^\infty(0,1)$ functions by setting $D(u) = +\infty$ for $u \in L^\infty(0,1) \setminus W^{1,2}(0,1)$.
The graph Dirichlet energy 
\begin{equation} \label{eq: graph dirichlet}
    D^{W_n}(u) = \begin{cases}
        \int_0^1\int_0^1 W_n(x,y) |u(x)-u(y)|^2 \,dx\,dy \;\; &\text{ if } \,\, u \in \mathcal{V}^n
        \\
        +\infty \;\; &\text{ if } \,\, u \in L^\infty((0,1))\setminus \mathcal{V}^n
    \end{cases}
\end{equation}
acts on graph functions and replaces the gradient term in~\eqref{eq: classical Dirichlet} with a finite difference.  

\begin{remark}\label{rmk: graph-cut equivalence}
    The graph Dirichlet energy is a generalization of the graph-cut functional \eqref{eq: graph-cut dirichlet form}: if we restrict graph functions to the range $\{-1,1\}$, then the graph Dirichlet energy \eqref{eq: graph dirichlet} is equal to the graph-cut functional \eqref{eq: graph-cut dirichlet form} (i.e., $D^{W_n}(u) = \text{Cut}(u)$ for all graph functions $u$).
\end{remark}


\subsection{Young measures and {weak-*} convergence of functions}\label{sec: ym}

In our study of $\Gamma$-convergence, we need to use Young measures \cite{dacorogna2006weak}, which extend the feasible set of the GL-minimization and TV-minimization problems to a space of measures. In general, the study of {the convergence of} functionals such as $\int_0^1 f(v(x)) \,dx$ (for a given continuous and bounded functional $f$) that act on $v \in L^\infty((0,1))$ presents a significant challenge. Because $L^\infty((0,1))$ is not a compact space, a sequence\footnote{In this sentence and many other times throughout this paper, we abuse the notation ``$\in$" to signify that each element of a sequence (rather than the sequence as a single object) is an element of a set.} $\{v_n\} \in L^\infty((0,1))$ may not have a limit in $L^\infty((0,1))$.  Therefore, a sequence $\{f_n(v_n)\}$ may not have a limit $f(v)$ in which $v$ is a limit of $v_n$. Our functionals ($\GL^W_\ep$, $\GL^{W_n}_\ep$, and so on) act on $L^\infty((0,1))$ and have the form $\int_0^1 f(v(x))\, dx$. To overcome this challenge, we extend these functionals to act on Young measures (rather than on $L^\infty$ functions). Intuitively, our use of Young measures accommodates the rapidly oscillating minima of the graphon GL functionals. The minima can oscillate arbitrarily rapidly, and Young measures provide an effective limit of such oscillating functions when they are acted on by continuous bounded functions in the integrals. In Section \ref{sec: examples}, we further discuss the need for Young measures. 

It is useful to review some definitions and properties that were presented in~\cite{braides2020cut}. 
\begin{definition}[Young measure]\label{def: ym}
    A \emph{Young measure} $\nu$ on $(0,1)\times \R$ is a family $\{\nu_x\}_{x \in (0,1)}$ of probability measures, which are
    parametrized by $x \in (0,1)$, such that the map 
\begin{equation}
    x\mapsto \int_\R f(\lambda) \, d\nu_x(\lambda)
\label{eq: def ym}
\end{equation}
is a Lebesgue-measurable function for every continuous and bounded function $f \in C^b(\R)$.
(We use the notation $C^b$ to signify continuous and bounded functions.)
\end{definition}

Intuitively, $\nu_x$ is a ``slice" of the Young measure $\nu$ at the value $x$. One can think of $\nu_x$ as analogous to the value of a function $u(x)$. (The value of a function at the point $x$ is the ``slice" of the function at $x$.) However, instead of assigning a value (as a function $u$ does) to each point (i.e., node) $x$ to designate the ``state" of $x$, a Young measure $\nu$ assigns a probability distribution $\nu_x$ to the point $x$. For Young measures on graphons, one can interpret the distribution $\nu_x$ as incorporating uncertainty into the state of node $x$. This is analogous to the way that a function $u_n$ on a graph has values $u_n(x)$ that give the state of node $x$.

Let $\mathcal{Y}((0,1),\R)$ denote the set of all Young measures on $(0,1) \times \R$. With the next definition, we see how $\mathcal{Y}((0,1),\R)$ extends the set of $L^\infty((0,1))$ functions. 

\begin{definition}[Young measure corresponding to a measurable function]
    A \emph{Young measure corresponding to a {Lebesgue-}measurable function} $u: (0,1) \ra \R$ is the family of delta measures 
\begin{equation} 
    \{\nu_x\}_{x\in (0,1)} = \{\delta_{u(x)}\}_{x\in (0,1)} \,. 
\label{eq: delta ym}
\end{equation}
{We refer to such measures as \emph{$\delta$-Young measures.}}

\end{definition}

With the definition of $\nu$ in equation~\eqref{eq: delta ym}, the map \eqref{eq: def ym} is the evaluation map $x \mapsto f(u(x))$. {This evaluation map is Lebesgue-measurable, as required for Definition \ref{def: ym}, because of the Lebesgue-measurability of $u$ and the continuity of $f$.} 

We now define the {weak-* topology (which is also known as ``weak-star topology" and the ``ultraweak topology")} on {the space} $L^\infty$.

\begin{definition}[Weak-* topology {on the space $L^\infty((0,1))$}]
A sequence {$\{v_n\} \in L^\infty((0,1))$} converges in the weak-* topology to {$v \in L^\infty((0,1))$} if 
\begin{equation*}
    \lim_{n\ra \infty} \int_0^1 {v_n(x)}g(x) \, dx = \int_0^1 {v(x)}g(x) \, dx
\end{equation*}
for any $g \in L^1((0,1))$. We write that $v_n \weakstarra v$ in $L^\infty((0,1))$. 
\end{definition}

\begin{definition}[Narrow convergence of Young measures]\label{definition narrow convergence}
    A sequence {$\{\nu^n\} \in \mathcal{Y}((0,1),\R)$} \emph{converges narrowly} to $\nu \in \mathcal{Y}((0,1), \R)$ if the map~\eqref{eq: def ym} converges {in the weak-* topology} in $L^\infty((0,1))$ for {all} 
    {continuous and bounded functions} $f \in C^b(\R)$. That is, for all $x \in (0,1)$, we have 
    \begin{equation}
	    \int_\R f(\lambda) \,d\nu^n_x(\lambda) \weakstarra \int_\R f(\lambda) \,d\nu_x(\lambda) \,. 
    \end{equation}
\end{definition}

\begin{lemma}[Narrow convergence of product Young measures{~\cite[Lemma 8]{braides2020cut}}]\label{lem: narrow convergence of products}
    Let $\{\nu^n\}$ be a sequence of Young measures that converges narrowly to $\nu \in \mathcal{Y}((0,1), \R)$. {{It is then guaranteed that} the {sequence ${\{\nu^n \otimes \nu^n\}}$ of product measures}} 
    on $(0,1)^2 \times \,\R^2$ converges narrowly to $\nu \otimes \nu \in \mathcal{Y}((0,1)^2,{\R^2})${. That is,} 
    {$d (\nu^n \otimes \nu^n)_{(x,y)}(\lambda,\mu) = d \nu^n_x(\lambda) \, d \nu^n_y(\mu)$ and}
    \begin{equation*}
        \iint_{\R^2} f(\lambda,\mu) \,d (\nu^n \otimes \nu^n)_{(x,y)}(\lambda,\mu)
        \weakstarra \iint_{\R^2} f(\lambda,\mu) \,d (\nu^n \otimes \nu^n)_{(x,y)}(\lambda,\mu)
    \end{equation*}
    for all $f \in C^b(\R^2)$, where $d (\nu \otimes \nu)_{(x,y)}(\lambda,\mu) = d \nu_x(\lambda) \, d \nu_y(\mu)$.
\end{lemma}

The following lemma states the key compactness property of $\mathcal{Y}((0,1), \R)$ that justifies the use of Young measures in our $\Gamma$-convergence results that involve graphon limits.

\begin{lemma}[Prohorov's Theorem {\cite[Theorem 9]{braides2020cut}}]  \label{prohorov's theorem}
 Let $\{u_n\}_{n\in \N}$ be a bounded sequence in $L^1((0,1))$, and let $\{\nu^n\}_{n\in \N}$ denote the sequence of corresponding Young measures (see equation~\eqref{eq: delta ym}). There then exists a subsequence $\{u_{n_k}\}_{k\in \N}$ and a Young measure $\nu$ such that $\nu^{n_k}$ converges narrowly to $\nu$ as $k \ra \infty$. 
\end{lemma}


\subsection{Graphon functions and functionals}\label{sec: graphon functionals}

We define the graphon GL functional  
\begin{equation} \label{eq: graphon GL}
    \GL^W_\ep (\nu) = \int_0^1\int_0^1 W(x,y) \int_{\R^2} |\lambda - \mu|^2 \,d\nu_x(\lambda) \,d\nu_y(\mu)\,dx\,dy  
	+ \frac{1}{\ep}\int_0^1 \int_\R \Phi(\lambda) \,d\nu_x(\lambda) \,dx \,. 
\end{equation} 
The first term of~\eqref{eq: graphon GL} is the graphon Dirichlet energy
\begin{equation}\label{eq: graphon dirichlet}
    D^W(\nu) = \int_0^1\int_0^1 W(x,y) \int_{\R^2} |\lambda - \mu|^2 \,d\nu_x(\lambda) \,d\nu_y(\mu) \,dx\,dy  \,,
\end{equation}
which is the graphon analogue of the graph Dirichlet energy~\eqref{eq: graph dirichlet}. For graphs $W_n$, we also use \eqref{eq: graphon dirichlet}, which is more general than the graph Dirichlet energy~\eqref{eq: graph dirichlet}. Equation~\eqref{eq: graphon dirichlet} reduces to~\eqref{eq: graph dirichlet} when $\nu$ is a Young measure that corresponds to a {measurable} function $u$.

The inner integral $\int_{\R^2} |\lambda - \mu|^2 \,d\nu_x(\lambda) \,d\nu_y(\mu)$ is an expectation of $|\lambda - \mu|^2$ with respect to the probability measures $\nu_x$ and $\nu_y$. This is a probabilistic analogue of the term $|u(x) - u(y)|^2$ in the graph GL functional. When the Young measures are $\nu_x = \delta_{u(x)}$ and $\nu_y = \delta_{u(y)}$, we recover $|u(x) - u(y)|^2$.

The graphon TV functional, which also acts on $\nu \in \mathcal{Y}((0,1),\R)$, is analogous to the graph TV functional, just as the graphon GL functional is analogous to the graph GL functional. The graphon TV functional is
\begin{equation}
    \TV^W(\nu) = \begin{cases} 2 \int_0^1\int_0^1 W(x,y) \int_{\R^2}|\lambda - \mu| \, d\nu_x(\lambda)\,d\nu_y(\mu) \,dx\,dy  &\text{ if } \,\, \nu \in \mathcal{Y}^b \\
    + \infty & \text{ if } \,\, \nu \in \mathcal{Y}((0,1),\R) \setminus \mathcal{Y}^b\,,
    \end{cases} 
    \label{eq: graphon TV}
\end{equation}
where $\mathcal{Y}^b$ denotes the set of Young measures {$\nu$} with support on $\{-1,1\}$ {{(i.e., the union of the supports of $\{\nu_x\}_{x \in (0,1)}$ is $\{-1,1\}$)}.


\section{\texorpdfstring{Sequential limit: $\ep$ then $n$ (i.e., $\GL^{W_n}_\ep \gamra \TV^{W_n} \gamra \TV^W$)}{Sequential limit: ep then n}}\label{sec: sequential ep then n}

In this section, we prove the limits (1) and (2) from Figure \ref{figure:convergences}. Limit (1) was already proven by van Gennip and Bertozzi \cite[Theorem 3.1]{van2012gamma}, and we state their result for completeness. Limit (2) was proven {for} square-lattice graphs in \cite[Theorem 4.3]{van2012gamma} and for point clouds in~\cite{garcia2016continuum}. We {prove limit (2) for} general sequences of {weighted, undirected, and simple} graphs. Our proof closely follows the proof of the main theorem of {Braides et al.}~\cite{braides2020cut}.


\subsection{\texorpdfstring{Limit (1): $\GL^{W_n}_\ep \gamra \TV^{W_n}$}{Limit (1)}}

We state two key results from~\cite{van2012gamma} that lead to
our limit (1). The cruxes of these results are that (1) $\GL^{W_n} \gamra \TV^{W_n}$ and that (2) the functionals $\GL^{W_n}$ and $\TV^{W_n}$ are defined on a compact set of graph functions. Consequently, the minimizers of $\GL^{W_n}$ converge to the minimizers of $\TV^{W_n}$. 
Propositions \ref{prop: limit 1 convergence} and \ref{prop: compactness limit 1} hold for all {undirected and weighted} graphs $W_n$.

\begin{proposition}[{$\Gamma$-convergence} {\cite[Theorem 3.1]{van2012gamma}}] \label{prop: limit 1 convergence} 
The graph GL functional~\eqref{eq: graph GL} $\Gamma$-converges to the graph TV functional~\eqref{eq: graph TV} as $\ep \ra 0$ with respect to $u_n \ra u$ in $\mathcal{V}^n$. That is,
   \begin{equation}
     \GL^{W_n}_\ep \gamra \TV^{W_n} \,. 
    \end{equation}
\end{proposition}

In concert with Proposition \ref{prop: limit 1 convergence}, the following compactness property of the set $\calV$ of functions guarantees that the minimizers of the $\Gamma$-converging functionals also converge.

\begin{proposition}[{Compactness} {\cite[Theorem 3.2]{van2012gamma}}]\label{prop: compactness limit 1}
     Let $\{\ep_n\}_{n=1}^\infty \in \R_+$ be a sequence such that $\ep_n \ra 0$ as $n\ra \infty$, and let $\{u_n\}_{n=1}^\infty \subset \mathcal{V}$ be a sequence for which there exists a constant $C > 0$ such that $\GL^{W_n}_{\ep_n} < C$ for all $n \in \N$. There then exists a subsequence $\{u_{n'}\}_{n'=1}^\infty \subseteq \{u_n\}_{n=1}^\infty$ and $u_\infty \in \mathcal{V}^n_b$ such that $u_{n'} \ra u_\infty$ as $n\ra \infty$. 
\end{proposition}


\subsection{\texorpdfstring{Limit (2): $\TV^{W_n} \gamra \TV^W$}{Limit (2)}} \label{sec: limit (2)}

We prove limit (2) using~\cite[Theorem 12]{braides2020cut}, which states that $I_n \gamra I$ as $W_n \sqra W$, where
\begin{align}
	I_n(\nu) &= \int_0^1\int_0^1 W_n(x,y) \int_{\R^2} f(\lambda, \mu) \,d\nu_x(\lambda) \,d\nu_y(\mu) \,dx\,dy \,, \\
	I(\nu) &= \int_0^1\int_0^1 W(x,y) \int_{\R^2} f(\lambda, \mu) \,d\nu_x(\lambda) \,d\nu_y(\mu) \,dx\,dy \,.
\end{align}	
Limit (2) follows directly by using the integrand $f(s,t) = |s - t|$ instead of $f(s,t) = |s - t|^2$, which is the integrand that {was} used in \cite[Theorem 12]{braides2020cut}.

\begin{theorem}\cite[Theorem 12]{braides2020cut}
Let the functions $f \in C^b((0,1)^2)$ be bounded and continuous, and let $u \in L^\infty((0,1))$ and $\nu \in \mathcal{Y}((0,1),\R)$. Finally, let $\{W_n\}_{n=1}^\infty$ be a {dense sequence of graphs,}
with $W_n \sqra W\in \mathcal{W}_0$. We then have that
\begin{equation}
	I_n \gamra I \text{ as } n\ra \infty
\end{equation}
with respect to the narrow convergence of measures in $\mathcal{Y}((0,1),\R)$. 
\end{theorem}

\begin{corollary}
With the choice $f(s,t) = |s - t|$, we have 
\begin{equation}
     \TV^{W_n} \gamra \TV^{W} \text{ as } n \ra \infty 
\end{equation}    
with respect to narrow convergence of {$\nu^n$} to $\nu$ in $\mathcal{Y}((0,1),\R)$. 
\end{corollary}

\begin{proposition}[Compactness]
Let $W_n \sqra W$, and let {$\{u_n\} \in \calV^n$} be a sequence of graph functions such that $ \TV^{W_n}(u_n) < M$ for all {$n \in \N$} and some {constant} $M > 0$. {There then} exists a convergence subsequence {$\{u_{n_k}\}$} {such} that the {sequence $\{\delta_{u_{n_k}(x)}\}_{x\in (0,1)}$ of corresponding $\delta$-Young measures} converges to a limiting Young measure $\nu$. 
\end{proposition}

\begin{proof}
    Because {the sequence} $\{u_n\}$ consists of graph functions, it {is} bounded in $L^1((0,1))$. {We then} apply Lemma \ref{prohorov's theorem} ({i.e.,} Prohorov's Theorem) to obtain the {desired} result. 
\end{proof}


\section{\texorpdfstring{Sequential limit: $n$ then $\ep$ (i.e., $\GL^{W_n}_\ep \gamra \GL^W_\ep \gamra \TV^W$)}{Sequential limit: n then ep}  }   \label{sec: sequential n then ep part n} 

In this section, we prove two novel limits. Our main result is limit (3), which extends \cite[Theorem 12]{braides2020cut}. We prove limit (4) only for $L^\infty$ graphons, and we discuss a scaling issue for $L^p$ graphons.

\subsection{\texorpdfstring{Limit (3): $\GL^{W_n}_\ep \gamra \GL^W_\ep$}{Limit (3)}}
\label{sec: limit (3)}

There are three key steps in our proof of limit (3).

We first show that the GL-minimization problem is well-posed in the space $L^\infty((0,1);[-1,1])$. That is, we show that $\arg\min_{L^\infty((0,1))} \GL^W_\ep \subset L^\infty((0,1);[-1,1])$. It then suffices to consider functions $u, u_n \in L^\infty((0,1);[-1,1])$. 

We then show that the graph Dirichlet energy $\Gamma$-converges to the graphon Dirichlet energy. Our proof includes ideas from the proof of~\cite[Lemma 11]{braides2020cut} for the $\Gamma$-convergence of graph-cut functionals, which are similar to Dirichlet energies, except that their domains consist of only
functions with finite-cardinality codomains.

Finally, we show that adding a double-well potential does not affect the $\Gamma$-convergence. {Because each GL functional is a sum of a Dirichlet energy and a double-well potential, this fact} yields $\Gamma$-convergence of the GL functionals.


\subsubsection{\texorpdfstring{Well-posedness {of $\GL^W_\ep$} in $L^\infty((0,1); [-1,1])$}{Well-posedness}}\label{sec: well-posedness} 

We show by contradiction that the function $u$ is bounded {in magnitude} by $M = 1$. Suppose that $M >1$, and let $u^M$ be the truncation of $u$ at $\pm M$. That is, 
\begin{equation} 
	u^M(x) = \begin{cases}
		M &\text{ if } \,\, u(x) > M \\
		u(x) &\text{ if } \,\, |u(x)| \leq M	\\
		-M &\text{ if } \,\, u(x) < -M \,.
\end{cases}
\end{equation}
We show that $\GL^W_\ep(u) \geq \GL^W_\ep(u^M)$ when $W$ is any graphon. This implies that the minimizer of $\GL^W_\ep$ is in $L^\infty((0,1); [-1,1])$.

We separately show the well-posedness in $L^\infty((0,1);[-1,1])$ of the Dirichlet energy and the double-well potential. Because the double-well potential $\Phi(s) = (s^2 - 1)^2$ increases as $s > 1$ increases and decreases as $s < -1$ decreases, we know that 
\begin{equation}
	\int_0^1\Phi(u) \geq \int_0^1 \Phi(u^M) 
\end{equation}	
{for $M >1$.} For the Dirichlet energy, let 
\begin{equation}
	S_M = \{x: u(x) \geq M\}
\end{equation}	
be the set of points $x \in (0,1)$ where $u^M$ and $u$ differ. To simplify our notation in this discussion, let
\begin{align*}
    g(x,y) &= |u(x) - u(y)|^2\,, \\
    g^M(x,y) &= |u^M(x) - u^M(y)|^2\,.
\end{align*}
We want to show that
\begin{align*}
    D^W(u) - D^W(u^M) &= \int_0^1 \int_0^1 W(x,y) \left(g(x,y) - g^M(x,y)\right) \,dx \, dy \\
    &= \Bigg(\int_{S_M} \int_{S_M^c} + \int_{S_M^c} \int_{S_M} + \int_{S_M} \int_{S_M} \Bigg) W(x,y)\left(g(x,y) - g^M(x,y)\right) \,dx \, dy 
\end{align*}
is nonnegative. The integrals over $S_M \times S_M^c$ and $S_M^c \times S_M$ are equal because the integrand is symmetric. Both of these integrals are equal to
\begin{align*}
    \int_{S_M} \int_{S_M^c} W(x,y) \Big(|u(x) - u(y)|^2 - |M - u(y)|^2\Big)  \,dx \,dy\,.
\end{align*}
Note that $|u(x) - u(y)|^2 - |M - u(y)|^2 \geq 0$ because the function $f(s) = |s - c|^2$ increases as $s$ increases for $s > c$. Consequently, the integral over $S_M \times S_M^c$ (and hence also the integral over $S_M^c \times S_M$) is nonnegative.
The integral over $S_M \times S_M$ is
\begin{align*}
  &  \int_{S_M} \int_{S_M} W(x,y) (g(x,y)- g^M(x,y)) \, dx\,dy \\
    &\qquad = \int_{S_M} \int_{S_M} W(x,y) \Big(|u(x) - u(y)|^2 - |M - M|^2\Big)  \,dx\,dy \\
    &\qquad = \int_{S_M} \int_{S_M} W(x,y) |u(x) - u(y)|^2 \,dx\,dy \geq 0 \,,
\end{align*}
where we use the nonnegativity of the integrand in the last step to obtain nonnegativity of the integral.

We conclude that $\int_0^1 \Phi(u(x)) \,dx \geq \int_0^1 \Phi(u^M(x)) \,dx$ and $D^W(u(x)) \geq D^W(u^W(x))$. Consequently, $\GL^W_\ep(u(x)) \geq \GL^W_\ep(u^M(x))$. Because $M > 1$ is arbitrary, it follows that the GL minimizer takes values in $[-1,1]$.


\subsubsection{{Proof} of limit (3)}

{The proof of limit (3) requires two {key} steps. First, we state and prove Lemma \ref{lemma: limit 3 pointwise}, which establishes $\Gamma$-convergence of the graph Dirichlet energy to the graphon Dirichlet energy in the large-graph limit. Second, we show {that} $\Gamma$-convergence holds even after adding a double-well potential to the graph and graphon Dirichlet energies. {We thereby obtain} $\Gamma$-convergence of the graph GL functional to the graphon GL functional in the large-graph limit.}

Theorem \ref{thm: dirichlet gamma convergence}, which guarantees $\Gamma$-convergence of the graph Dirichlet energy \eqref{eq: graph dirichlet} to the graphon Dirichlet energy \eqref{eq: graphon dirichlet}, resembles~\cite[Theorem 12]{braides2020cut}, but it extends it in two important ways. First, it extends the domain from finite-codomain $u_n$ to $u_n \in L^\infty((0,1))$. Second, it extends graphons from $W \in L^\infty((0,1)^2)$ to $W \in L^p((0,1)^2)$. 

Lemma \ref{lemma: limit 3 pointwise}
provides the foundation for the proof of Theorem \ref{thm: dirichlet gamma convergence}, which we use in turn to prove limit (3) (see Corollary \ref{cor: gl gamma convergence}). 

\begin{lemma}\label{lemma: limit 3 pointwise} 
Let $W_n \in L^p((0,1)^2)$ for $p \geq 1$, and suppose that $W_n \sqra W$. Let $f \in C^b(\R^2)$ be a continuous and bounded function, and define the functional $I_n : \calY((0,1),\mathbb{R}) \to [0,\infty)$ as
\begin{equation*}
    I_n(\nu^n) = \int_0^1\int_0^1 W_n(x,y) \int_{\R^2} f(\lambda,\mu) \,d\nu^n_x (\lambda) \,d\nu^n_y(\mu) \,dx \, dy \,.
\end{equation*}
Let {$\{u_n\} \in \mathcal{V}^n$} be a sequence of graph functions such that $\sup_n \|u_n\|_\infty < \infty$. We then have that the {sequence $\{ \nu_x^n = \delta_{u_n(x)}\}_{x\in (0,1)} \subset \mathcal{Y}((0,1),\mathbb{R})$ of corresponding Young measures is} precompact (i.e., its closure is compact) in the narrow topology. Moreover, any subsequence $\{ \nu_x^{n_k} \}$ of $\{\nu^n_x\}$ with a corresponding limit point $\nu$ satisfies
\begin{equation}
	I_n(\nu^{n_k}) \rightarrow I(\nu)
\end{equation}	
pointwise, where the functional $I : \mathcal{Y}((0,1),\mathbb{R}) \to [0,\infty)$ is 
\begin{equation*}
    I(\nu) = \int_{(0,1)^2} W(x,y) \int_{\R^2} f(\lambda,\mu) \,d\nu_x(\lambda) \,d\nu_y(\mu)  \,dx \, dy \,.
\end{equation*}
\end{lemma}

\begin{proof}
To prove this result, we use the triangle inequality to break $|I_n(\nu^n) - I(\nu)|$ into two parts. {One part} converges to $0$ due to {the} weak convergence of $g_n$ to $g$, and the other part converges to $0$ due to the cut convergence $W_n \sqra W$.

By {Prohorov's Theorem ({see} Lemma \ref{prohorov's theorem})}, any sequence $\{\nu^n\}$ of $\delta$-Young measures that corresponds to a sequence $\{u_n\} \in L^\infty((0,1))$ has a subsequence $\{\nu^{n_k}\}$ that converges narrowly to a Young measure $\nu$. For the remainder of our paper, we simplify our notation by using $\{\nu^n\}$ to denote the subsequence $\{\nu^{n_k}\}$.
We denote the innermost integrals of the functionals by 
\begin{align*}
    g_n(x,y) &= \int_{\R^2} f(\lambda,\mu) \,d\nu^n_x (\lambda) \,d\nu^n_y(\mu) = f(u_n(x),u_n(y)) \,, \\
    g(x,y) &= \int_{\R^2} f(\lambda,\mu) \,d\nu_x (\lambda) \,d\nu_y(\mu) \,,
\end{align*}
and we can then write $I_n(\nu^n) = \int_{(0,1)^2}W_n(x,y) g_n(x,y) \,dx \, dy$ and $I(\nu) = \int_{(0,1)^2} W(x,y) g(x,y) \, dx \, dy$. 
The triangle inequality gives
\begin{align*}
    |I_n(\nu^n) - I(\nu)| &= \left|\int_0^1\int_0^1 W_ng_n - Wg \,dx\,dy\right| \\
	    &\leq \left| \int_0^1\int_0^1 (W_n - W)g_n \,dx\,dy\right| + \left| \int_0^1\int_0^1 W(g_n - g)  \,dx\,dy\right| \\
	    &\equiv ({\mathrm{I}}) + ({\mathrm{II}}) \,. 
\end{align*}

We show that {${(\mathrm{II})} \ra 0$} using weak convergence of $g_n$ to $g$. The narrow convergence $\nu^n \weakstarra \nu$ implies that {$g_n$ converges in the weak-* topology to $g$} in $L^{\infty}((0,1)^2)$ by {Lemma}~\ref{lem: narrow convergence of products}. That is,
\begin{equation*}
    \int_{\mathbb{R}^2} f(\mu,\lambda) \, d \nu_x^n(\mu) \, d \nu_y^n(\lambda) \weakstarra \int_{\mathbb{R}^2} f(\mu,\lambda) \, d \nu_x(\mu) \, d \nu_y(\lambda) \, \text{ in }  L^{\infty}((0,1)^2,\mathbb{R})\,. 
\end{equation*} 
Furthermore, $W \in L^p((0,1)^2) \subset L^1((0,1)^2)$, so the definition of weak convergence in $L^\infty$ implies that $(\mathrm{II}) \ra 0$.

We now show that {$(\mathrm{I}) \ra 0$} because $W_n \sqra W$. {To do this}, we approximate $g_n$ by polynomials to obtain a sum of terms that resembles the definition of cut convergence. This yields an expression that has the same form as the right-hand side of equation~\eqref{eq: cut norm 2}.

Let {$\{P_k\}_{k=1}^\infty : \R^2 \to \R$} be a sequence of polynomial functions such that
\begin{equation*}
    |f(a,b) - P_k(a,b)| \to 0 \,\text{ as }\, k \to \infty \,\text{ uniformly on }\, [-1,1]^2 \,. 
\end{equation*}
Such a polynomial exists because the set of polynomials is dense in $L^\infty([-1,1])$. Using the triangle inequality, we obtain
\begin{equation*}
    \begin{split}
        {(\mathrm{I})} &\leq \left| \int_{(0,1)^2} (W_n - W)\big( f(u_n(x),u_n(y)) - P_k(u_n(x),u_n(y)) \big) \,dy\, dx \right| \\
            &\qquad + \left| \int_{(0,1)^2} (W_n - W) P_k(u_n(x),u_n(y)) \,dy\,dx \right| \\
        &\leq \sup_{(a,b) \in [-1,1]^2 } |f(a,b) - P_k(a,b)| \, \left(  \| W_n \|_{L^1((0,1)^2)} + \| W \|_{L^1((0,1)^2)} \right)  \\
        &\qquad + \left| \int_{(0,1)^2} (W_n-W) P_k(u_n(x),u_n(y)) \,dy\,dx \right| \\
        &\leq C \sup_{(a,b) \in [-1,1]^2 } |f(a,b) - P_k(a,b)| \\
        &\qquad + \left| \int_{(0,1)^2} (W_n - W) P_k(u_n(x),u_n(y)) \,dy\,dx \right|
    \end{split}
\end{equation*}
because both the graphon $W$ and the sequence $\{ W_n \}$ are bounded in $L^1((0,1)^2)$.

For any polynomial $P_k(a,b) = \sum_{i,j = 1}^k \alpha_{ij} a^i b^j$, we use the boundedness of $\|u_n\|_\infty$ 
(specifically, we use $\| u_n \|_\infty \leq 1$, which we showed in Section \ref{sec: well-posedness}) to obtain
\begin{align*}
    \left| \int_0^1\int_0^1 (W_n - W) P_k(u_n(x),u_n(y)) \, dx\,dy \right| &= \left| \int_0^1\int_0^1 (W_n - W) \sum_{i,j=1}^k  \alpha_{ij} u_n^i(x) u_n^j(y) \,dx\,dy \right|  \\
    &= \left| \sum_{i,j = 1}^k \alpha_{ij} \int_0^1\int_0^1 (W_n - W) u_n^i(x) u_n^j(y) \,dx\,dy \right| \\
    &\leq k \, {{\max_{i,j}}\left\{\alpha_{ij} \cdot \left|\int_0^1\int_0^1 (W_n - W) u_n^i(x) u_n^j(y) \,dx\,dy \right|\right\}} \\
    &\leq C(P_k) \|W_n - W\|_\square \,.
\end{align*}

In summary,
\begin{equation*}
    {(\mathrm{I})} \leq C \left(  \sup_{(a,b) \in [-1,1]^2 } |f(a,b) - P_k(a,b)| + \|W_n - W\|_\square \right) \,,
\end{equation*}
where the constant $C$ is independent of $n$. Choosing $k$ sufficiently large and then letting $n \to \infty$ implies that {$(\mathrm{I}) \to 0$}. 
\end{proof}

\begin{corollary}\label{cor: dirichlet convergence} 
Let $W_n \sqra W$, and let {$\{u_n\} \in \mathcal{V}^n$} be a sequence of graph functions such that $\sup_n \|u_n\|_\infty < \infty$. We then have that the {sequence $\{\nu^n_x\} \subset \mathcal{Y}((0,1),\R)$ of corresponding Young measures} is precompact in the narrow topology. Moreover, any subsequence $\{\nu^n_x\}$ and any limit point $\nu$ {satisfy}
 \begin{equation} \label{cor2}
    	D^{W_n}(\nu^n) \ra D^W(\nu) 
\end{equation}	
pointwise, where $D^{W_n}$ and $D^W$ are defined {in \eqref{eq: graph dirichlet} and \eqref{eq: graphon dirichlet}, respectively.}
\end{corollary}

\begin{proof}
    Choose $f(s,t) = |s - t|^2$ in Theorem \ref{lemma: limit 3 pointwise}. 
\end{proof}

Corollary \ref{cor: dirichlet convergence} extends~\cite[Lemma 11]{braides2020cut} by allowing $\nu$ to be any Young measure and allowing {$\nu^n = \{\delta_{u_n(x)}\}_{x\in (0,1)}$ to be any $\delta$-Young measure {that corresponds} to some} $u_n \in L^\infty((0,1))$. In~\cite[Lemma 11]{braides2020cut}, $\nu^n$ must have support on a finite set of values, where the number of values is independent of $n$. The closure of the set of such Young measures is a strict subset of $\calY((0,1),\R)$. On the contrary, we allow $\nu^n$ to be any $\delta$-Young measure. As we note {below} in the proof of Theorem \ref{thm: dirichlet gamma convergence}, the closure of the set of $\delta$-Young measures in $\calY((0,1),\R)$ is the set $\calY((0,1),\R)$. See \cite[Proposition 8]{valadier1990young} for a proof of this fact.

\begin{theorem}\label{thm: dirichlet gamma convergence} 
Under the same assumptions as in Corollary \ref{cor: dirichlet convergence}, we have
\begin{equation}
	D^{W_n} \gamra {D^W} \,,
\end{equation}	 
where we take $\Gamma$-convergence with respect to narrow convergence of Young measures. 
\end{theorem}

\begin{proof}
To prove $\Gamma$-convergence, it suffices to prove the following two statements.
\begin{enumerate}[(i)]
    \item For every $\{\nu^n\}_{n\in \N}$ such that $\nu^n \ra \nu$ narrowly, we have $D(\nu) \leq \liminf_{n} D^{W_n}(\nu^n)$. 
    \item There exists a sequence $\{\nu^n\}_{n\in \N}$ that converges narrowly to $\nu$ with 
    \\$D(\nu) \geq \limsup_{n} D^{W_n}(\nu^n)$. 
\end{enumerate}

Statement (i) follows from Corollary \ref{cor: dirichlet convergence} because pointwise convergence holds for all $\nu^n \ra \nu$ narrowly. To prove statement (ii), it suffices to show that there exists a sequence {$\{\nu^n\} = \{\delta_{u_n(x)}\}$} that converges narrowly to some $\nu$ for any $\nu \in \mathcal{Y}((0,1))$. Once we obtain this sequence, Corollary \ref{cor: dirichlet convergence} yields inequality (ii). The existence of such a sequence {is called the ``fundamental theorem of Young measures"}, and it follows by the denseness {(with respect to narrow convergence)} of Young measures that correspond to measurable functions in the set of all Young measures~\cite[Proposition 8]{valadier1990young}. 
\end{proof}

\begin{remark}\label{direct}
Our proof of statement (ii) is more direct than the analogous proof by Braides et al.~\cite[Theorem 12]{braides2020cut}, who constructed $\nu^n$ carefully from a sequence of finite-valued functions $u_n$.
\end{remark}

Proposition \ref{prop: double-well convergence} {guarantees} that a sequence of double-well potentials on graph functions converges to a double-well potential on a Young measure. 
In Corollary \ref{cor: gl gamma convergence}, we use this result in concert with Theorem \ref{thm: dirichlet gamma convergence} to show that $\GL^{W_n}_\ep \gamra \GL^W_\ep$. 

\begin{proposition} \label{prop: double-well convergence}
    Let {$\{u_n\} \in L^\infty((0,1))$} be a sequence of functions that are constant on the intervals $I_i$
    for $i \in [n]$, and let $\Phi: \R \ra \R$ be continuous and bounded. Let $\nu$ be the Young measure that is the narrow limit of the sequence {$\{\nu^n\}$} of Young measures of the form $\nu^n_x = \delta_{u_n(x)}$.
    We have that
\begin{equation}
    \int_0^1 \Phi(u_n(x)) \,dx \ra \int_0^1 \int_{\R} \Phi(\lambda) \, d\nu(\lambda) \,dx  
 \end{equation}   
    pointwise. 
\end{proposition}

\begin{proof}
    The limit $\nu$ exists by Lemma \ref{prohorov's theorem} ({i.e.,} Prohorov's Theorem). 
    By the definition of narrow convergence, {the fact that $\nu^n \weakstarra \nu$ implies that} 
 \begin{equation} \label{narrow conv phi_n}
    \Phi(u_n(x)) = \int_\R \Phi(\lambda) \,d\nu^n_x(\lambda) \weakstarra \int_{\R} \Phi(\lambda) \,d\nu_x(\lambda) \; \text{ in }L^\infty((0,1)) \,.  
\end{equation}    
For any test function $\psi(x) \in L^1((0,1))$, the definition of weak-$*$ convergence implies that the sequence of integrals $\int_0^1 \Phi(u_n(x)) \psi(x) \,dx = \int_0^1 \int_\R \Phi(\lambda) \, d\nu^n_x(\lambda) \psi(x) \, dx$ converges to 
\\
$\int_0^1 \int_\R \Phi(\lambda) d\nu_x(\lambda) \psi(x) \,dx$ as $n\ra\infty$. Using the test function $\psi \equiv 1$ yields the desired result.

\end{proof}

The $\Gamma$-convergence of the graph GL functional \eqref{eq: graph GL} {to the graphon limit as $n \to \infty$} follows from the $\Gamma$-convergence of the graph Dirichlet energy \eqref{eq: graph dirichlet} (see Corollary \ref{cor: dirichlet convergence}) and the pointwise convergence of the double-well potential (see Proposition \ref{prop: double-well convergence}). We state this convergence result formally in the next corollary.

\begin{corollary} \label{cor: gl gamma convergence}
    Under the same assumptions as in Corollary \ref{cor: dirichlet convergence}, we have
 \begin{equation} 
    \GL^{W_n}_\ep \gamra \GL^W_\ep
  \end{equation}   
    as $n \ra \infty$ with respect to narrow convergence of the Young measures.
\end{corollary}

\begin{proof}
Let $\GL^{W_n}_\ep= F_n + G_n$ and $\GL^W_\ep= F + G$, where $F_n$ and $F$ denote the graph and graphon Dirichlet energies, respectively, and $G_n$ and $G$ correspond to the graph and graphon double-well energies, respectively.
Given the $\Gamma$-convergence (see Corollary \ref{cor: dirichlet convergence}) of the graph Dirichlet energy~\eqref{eq: graph dirichlet} (that is, $F_n \gamra F$) and the fact that $G_n \ra G$ (i.e., the pointwise convergence of the double-well potential that we proved in Proposition \ref{prop: double-well convergence}), we invoke the fact that $\Gamma$-convergence still holds under a continuous perturbation (see \cite[Remark 2.2]{braides2006handbook}).
Consequently, we obtain $F_n + G_n \gamra F + G$ with respect to narrow convergence of Young measures. 
\end{proof}

\medskip


\subsection{\texorpdfstring{The issue of $\ep$-scaling}{The issue of ep scaling}}\label{sec: scaling}

A key property of the classical GL functional~\eqref{eq: classical GL} is its $\Gamma$-convergence to the TV functional in the $\ep \ra 0$ limit \cite{modica1977esempio}. Consider
\begin{equation} 
\label{classical GL replicated}
	 \GL_\ep(u) = \ep \int_0^1 |\nabla u|^2 \,dx + \frac{1}{\ep} \int_0^1 \Phi(u(x)) \,dx \,, 
\end{equation}	 
which is equation~\eqref{eq: classical GL} (but we show it again for convenience).
As $\ep$ shrinks, the double-well term becomes larger due to its prefactor ${1}/{\ep}$. This larger contribution from the double-well term encourages narrower regions for $u$ to jump between $-1$ and $1$. However, steeper jumps of $u$ contribute more to the Dirichlet energy $\int_0^1 |\nabla u|^2 \,dx$. As the interface size $\ep$ shrinks to $0$, the contribution from each jump grows to $\infty$. The prefactors that ensure that the Dirichlet energy and double-well potential remain $O(1)$ as $\ep \ra 0$ are $\ep$ and $1/\ep$, respectively. One can see {the associated scaling limits} by substituting $x \mapsto x/\ep$ into the classical GL functional~\eqref{eq: classical GL}. Compare~\eqref{classical GL replicated} to the graph GL functional
\begin{equation}\label{eq: graph gl functional 2}
	\GL^{W_n}_\ep = \int_0^1\int_0^1 W_n(x,y)|u(x)-u(y)|^2 \, dx \, dy + \frac{1}{\ep} \int_0^1 \Phi(u(x))\,dx  \,,
\end{equation}	
which is equation~\eqref{eq: graph GL}. The graph GL functional~\eqref{eq: graph gl functional 2} does not require the prefactor $\ep$ in the graph Dirichlet energy because a graph inherently has no infinitesimal spatial {limit. Each} jump of $u$ from $-1$ to $1$ {(and vice versa)} contributes a finite amount to the {graph} Dirichlet energy, so both {terms of the graph GL functional remain} $O(1)$ even when $\ep \ra 0$. As with the classical GL functional, the} double-well potential is a penalty term that enforces $u$ to be binary.

It is natural to ask how one should scale the graphon GL functional $\GL^W_\ep$, which is similar {to \eqref{eq: graph gl functional 2} in the sense that} one computes differences of $u$ (rather than derivatives of $u$). Note that $W(x,y) |u(x) - u(y)|^2 = 4 W(x,y)$ when $u(x) = 1$ and $u(y) = -1$ (or vice versa). Because $4W(x,y)$ is finite when $W \in L^\infty((0,1)^2)$, one does not need an $\ep$ prefactor {for} $L^\infty$ graphons. 

If the graphon Dirichlet energy \eqref{eq: graphon dirichlet} is unbounded, as is the case for $L^p$ graphons, the scaling {rate} in $\ep$ depends on the choice of $W$. The reason for this is that it is necessary to scale the unbounded graphon Dirichlet energy $D^W$ at an appropriate rate to ensure that both the Dirichlet energy and the double-well potential remain $O(1)$. The more singular $W$ is, the faster one needs to scale down the graphon Dirichlet energy (and vice versa). Formalizing the relationship between the singularity strength of $W$ and the $\ep$-scaling is an interesting topic for future work.


\subsection{\texorpdfstring{Limit (4): $\GL^W_\ep \gamra TV^W$ as $\ep \ra 0$ for {$W \in L^\infty$}}{Limit (4)}} \label{sec: limit (4)}

We prove a version of the classical limit $\GL_\ep \gamra \TV$ for graphon GL and graphon TV functionals. We consider only $L^\infty$ graphons because it is difficult to determine the correct $\ep$-scaling for general $L^p$ graphons.

\begin{theorem}\label{thm: graphon gl convergence ep limit}
Suppose that $W \in L^\infty((0,1)^2)$. {As $\ep \ra \infty$, we then have}
    \begin{equation}
     	\GL^W_\ep \gamra \TV^W 
     \end{equation}
     {with respect to narrow convergence of Young measures.} 
\end{theorem}

To prove Theorem \ref{thm: graphon gl convergence ep limit}, we follow a strategy that resembles the proof of \cite[Theorem 3.1]{van2012gamma}. Suppose that {a sequence {$\{\nu^n\} \in \mathcal{Y}((0,1),\R)$} of Young measures} converges narrowly to $\nu \in \mathcal{Y}((0,1),\R)$. Lemma \ref{lemma: double-well gamma convergence} shows that the double-well potential $\Gamma$-converges either to $0$ or to $\infty$ as $\nu^n \ra \nu$ narrowly.

\begin{lemma} \label{lemma: double-well gamma convergence}
    Let $\Phi: \R \ra \R$ be a continuous and bounded potential function. Consider the functionals $E_\ep$ and $E_0$ on $\mathcal{Y}([0,1],\R)$ that are given by
    \begin{equation}\label{graphon double well}
 	   E_\ep(\nu) = 
        \begin{cases}
        \displaystyle
            \frac{1}{\ep} \int_0^1 \int_\R \Phi(\lambda)\,d\nu_x(\lambda) \, dx &\text{ if } \,\, {\nu \in \mathcal{Y}([0,1],\{-1,1\})} \\
        + \infty & {\text{ otherwise}}
        \end{cases}
    \end{equation}	   
    and
    \begin{equation}
    E_0(\nu) = 
        \begin{cases}
        0 &\text{ if } \,\, {\nu \in \mathcal{Y}([0,1],\{-1,1\})} \\
        + \infty &\text{ otherwise}\,, 
        \end{cases}
    \end{equation}
where $\mathcal{Y}([0,1],\{-1,1\}) \subset \mathcal{Y}([0,1],\R)$ is the set of Young measures with support on $\{-1,1\}$. As $\ep \ra 0$, we have
\begin{equation}
        E_\ep \gamra E_0 
    \label{eq: double-well gamma convergence}
\end{equation}      
    with respect to narrow convergence of {Young measures}. 
\end{lemma}

\begin{proof}[Proof of Lemma \ref{lemma: double-well gamma convergence}]
    Consider sequences $\{\ep_n\}_{n\in \N}$ and $\{\nu^n\}_{n\in \N}$ such that $\ep_n \ra 0$ and $\nu^n \ra \nu$ narrowly for some $\nu \in \mathcal{Y}([0,1],{\mathbb{R}})$. We verify the {lim inf} inequality (i) and the {lim sup} inequality (ii) of Definition \ref{def: gamma convergence}.

    {For} $\nu \in \mathcal{Y}([0,1],\{-1,1\})$, it follows that $E_0(\nu) = 0$. {Therefore, the lim inf} inequality $E_0(\nu) \leq$ $\liminf\limits_{n\ra\infty} E_{\ep_n}(\nu^n)$ 
    holds because {$E_{\ep_n}$} is nonnegative. {For} $\nu \in \mathcal{Y}([0,1],\R) \setminus \mathcal{Y}([0,1],\{-1,1\})$, we will show that $\nu^n \not\in \mathcal{Y}([0,1],\{-1,1\})$ for sufficiently large $n$. 
    {{To do this, we prove the contrapositive statement.} Suppose that $\nu^n \in \calY([0,1], \{-1,1\})$ for arbitrarily large $n$. The limiting Young measure $\nu$ then cannot have support on any point $y \in \R \setminus \{-1,1\}$. Consider a continuous and bounded function $f$ with support on a small interval around {the point} $y$ that does not include the points $1$ or $-1$.
    The narrow convergence of $\nu^n$ to $\nu$ implies that $\int_\R f(\lambda) d\nu^n_x(\lambda) \weakstarra \int_\R f(\lambda) \,d\nu_x(\lambda)$. However,
    $\int_\R f(\lambda) \,d\nu_x(\lambda) \geq 0$ for all $x$ and $\int_\R f(\lambda) d\nu^n_x(\lambda) = 0$ for all $x$. }
    Therefore, {for} $\nu \in \mathcal{Y}([0,1],\R) \setminus \mathcal{Y}([0,1],\{-1,1\})$, there must be some $x$ such that $\nu^n_x$ has support in $\R \setminus \{-1,1\}${. Therefore,} $\int_\R \Phi(\lambda)\,d\nu^n_x(\lambda) > 0$ for that $x${, so
    the lim inf inequality holds. Consequently,} 
    \begin{equation}
        \liminf_{n\ra\infty} E_{\ep_n}(\nu^n)\geq  {\int_0^1} \liminf_{n\ra\infty} \frac{1}{\ep_n}   \int_\R \Phi(\lambda)\,d\nu^n_x(\lambda) \, dx = \infty = E_0(\nu) \,. 
    \end{equation}   
The first inequality holds {by} Fatou's {Lemma}. The first equality holds because $\ep_n \ra 0$ and the integrand is uniformly bounded away from {$0$}. The second equality holds because $\nu^n \not\in \mathcal{Y}([0,1],\{-1,1\})$. We now verify the {lim sup} inequality $\limsup_{n\ra\infty} E_{\ep_n}(\nu^n) \leq E_0(\nu)$. Suppose that $\nu \in \mathcal{Y}([0,1],\{-1,1\})$. We choose $\nu^n = \nu$ for all $n$, so $E_0(\nu) = 0 = \limsup_{n\ra \infty} E_{\ep_n}(\nu^n)$. {With} $\nu \in \mathcal{Y}([0,1],\R)\setminus \mathcal{Y}([0,1],\{-1,1\})$, {it follows that} $E_0(\nu) = \infty$ and thus that the {lim sup} inequality $\limsup_{n\ra\infty} E_n (\nu_n) \leq E_0(\nu)$ holds for any sequence {$\{\nu^n\}$} that converges narrowly to $\nu$. 
\end{proof}

We now prove Theorem \ref{thm: graphon gl convergence ep limit}.

\begin{proof}[Proof of Theorem \ref{thm: graphon gl convergence ep limit}]

To prove the desired result, we use the fact that $\Gamma$-convergence persists under a continuous perturbation. Note that $\GL^W_\ep(\nu) = D^W(\nu) + E_\ep(\nu)$, where $D^W$ is the graphon Dirichlet energy \eqref{eq: graphon dirichlet} and $E_\ep$ is the double-well potential \eqref{graphon double well}. {The key ingredients in our proof are} the guarantee 
that the double-well potential $\Gamma$-converges (see Lemma \ref{lemma: double-well gamma convergence}) and {the fact} that the Dirichlet energy~\eqref{eq: graphon dirichlet} is continuous in $\nu$. {The Dirichlet energy} is continuous because $\nu^n$ converges narrowly to $\nu$, which implies (by Definition~\eqref{definition narrow convergence}) for $W \in L^1$ that
\begin{align*}
&  \int_0^1\int_0^1 W(x,y) \int_{\R^2} |\lambda-\mu|^2 \,d\nu^n_x(\lambda)\,d\nu^n_y(\mu) \,dx\,dy \\
& \ra \int_0^1 \int_0^1 W(x,y) \int_{\R^2} |\lambda-\mu|^2 \,d\nu_x(\lambda)\,d\nu_y(\mu) \,dx\,dy\,.
\end{align*}
Therefore, 
\begin{equation*}
    \GL^W_\ep = D^W + \frac{1}{\ep} \int_0^1 \int_\R \Phi(\lambda)\,d\nu_x(\lambda) \, dx 
\end{equation*}	
is a continuous perturbation of the double-well potential. Because $\Gamma$-convergence is stable under continuous perturbations, equation~\eqref{eq: double-well gamma convergence} implies that
\begin{equation}
    \label{eq: dirichlet plus indicator}
    \GL^W_\ep \; \gamra \; D^W + \begin{cases}
        0 &\text{ if } \,\, \nu \in \mathcal{Y}^b \\
        + \infty &\text{ otherwise} 
    \end{cases}
\end{equation}
as $\ep \ra 0$.

The right-hand side of~\eqref{eq: dirichlet plus indicator}  is {equal to} $\TV^W(\nu)$. We quickly verify this statement. If $\nu \in \mathcal{Y} \setminus \mathcal{Y}^b$, then both the $\Gamma$-limit in~\eqref{eq: dirichlet plus indicator} and $\TV^W(\nu)$ are $\infty$. If $\nu \in \mathcal{Y}^b$, then $\lambda$ and $\mu$ can only take the values $\pm 1$. Therefore, $\lambda$ and $\mu$ can only {have support in $\{\pm 1\}$}, so $|\lambda - \mu|^2$ equals either $2|\lambda - \mu|$ or $0$. 
\end{proof}


\section{{GL minimizers} for several examples}\label{sec: examples}

It is informative to {determine} the GL minimizers for some families of graphons. We characterize the minimizers of several graph GL functionals, and we illustrate the resulting Young measure in the graphon limit. Taking the $\ep \ra 0$ limit of a GL minimizer then allows us to infer the associated $\TV$ minimizer.

For more general situations than our simple examples, it is difficult to obtain analytical characterizations of graph and graphon GL minimizers. One can seek more general Young-measure minimizers using numerical approximations; see \cite{nicolaides1993computation, carstensen2000numerical} for possible approaches. We leave such endeavors to future work.

The Young measure $\nu_x = \delta_{1}$ for all $x$ (i.e., the constant function $u(x) = 1$) is a 
trivial minimizer of the graphon GL functional \eqref{eq: graphon GL} for all graphons $W$. Similarly, $\nu_x = \delta_{-1}$ is also a minimizer for all $W$. These are minimizers because the double-well potential $\Phi(s)$ is $0$ for $s = \pm 1$ and the Dirichlet energy term is $0$ for constant functions. {The constant function $u\equiv 1$ (which corresponds to the Young measure $\nu_x \equiv \delta_{1}$}) is also a trivial minimizer of the graph functional \eqref{eq: graph GL}. 

The trivial minimizer is the only minimizer of the GL functional because the GL functional is nonnegative and equals $0$ only for the trivial minimizer. To make the minimization problem nontrivial, we use {a mean-value constraint.}
For the graphon GL-minimization problem, we thus {impose the mean value}
\begin{equation}\label{eq: graphon mean-value constraint}
    \int_0^1 \int_{\R} \lambda \, d\nu_x(\lambda) \,dx = c
\end{equation}
for a {specified} constant $c \in (-1,1)$. The analogous {constraint} for the graph GL-minimization problem is 
\begin{equation}\label{eq: graph mean-value constraint}
    \frac{1}{n} \sum_{i = 1}^n u_i = c \,. 
\end{equation}
This constraint entails that the graph function $u$ has a mean value of $c$. The double-well potential $\Phi$ forces $u$ to take values near $+1$ and $-1$, and the mean-value constraint prevents $u$ from settling at only one of the values $+1$ or $-1$. Consequently, $u$ exhibits a \textit{phase separation}, where the values $+1$ and $-1$ correspond to two different ``phases"~\cite{du2020phase}. 
We consider $c \in (-1,1)$ in Section \ref{sec: constant graphon} and consider $c = 0$ in Section \ref{sec: sbm}.

In this section, we denote graph functions by $u$ (rather than by $u_n$, as we did previously). Additionally, we denote both functions and Young measures on graphons by their Young measure counterpart $\nu$.


\subsection{The constant graphon} \label{sec: constant graphon}

The constant graphon is the large-graph limit of the complete graph (for which $W \equiv 1$), Erd\H{o}s--R\'{e}nyi (ER) graphs (for which $W \equiv p \in (0,1)$)\cite[Section 10.1]{lovasz2012large}, and some growing preferential-attachment graphs (including Barab\'{a}si--Albert (BA) graphs, for which $W \equiv {p \in (0,1)}$)~\cite[Example 11.44 and Proposition 11.45]{lovasz2012large}. 

Consider the graph GL functional on the $n$-node complete graph. Let $A^{(n)}_{ij} = p$ for all $i$ and $j$ in \eqref{eq: graph GL}. This leads to
\begin{align*}
    \GL^{W_n}_\ep (u) &= \frac{1}{n^2} \sum_{i,j = 1}^n p (u_i - u_j)^2 + \frac{1}{\ep n} \sum_{i = 1}^n \Phi (u_i) 
    \\
    &= \frac{p}{n^2} \left( 2n\sum_{i = 1}^n u_i^2 - 2\sum_{i,j = 1}^n u_i u_j \right) + \frac{1}{\ep n} \sum_{i = 1}^n (u_i^4 - 2u_i^2 + 1)
    \\
    &= \frac{1}{\ep n}  \sum_{i = 1}^n  \left(u_i^4 -  (2 - 2\ep p )  u_i^2\right) - \frac{2{p}}{n^2} \left(\sum_{i = 1}^n u_i\right)^2 + \frac{1}{\ep}
        \\
    &= \frac{1}{\ep n}  \sum_{i = 1}^n  (u_i^2 -  (1 - \ep p))^2 - \frac{2p}{n^2} \left(\sum_{i = 1}^n u_i \right)^2 - \frac{1}{\ep} (1 - \ep p)^2 + \frac{1}{\ep} \,. \numberthis 
    \label{eq: const discrete fnl}
\end{align*}

Fix $\ep < 1$. {By removing constant terms, we see that} minimizing \eqref{eq: const discrete fnl} subject to the {mean-value constraint} \eqref{eq: graph mean-value constraint} is equivalent to minimizing the energy
\begin{equation}\label{eq: equivalent const gl fnl}
   E_n(v):= \sum_{i=1}^n \left( v_i^2 - 1\right)^2 \,, 
\end{equation}
where 
\begin{equation}\label{eq: change of vars}
    v_i = \frac{u_i}{\sqrt{{1 - \ep p}}} \,.
\end{equation} 
With the change of variables \eqref{eq: change of vars}, the {mean-value constraint} \eqref{eq: graph mean-value constraint} becomes 
\begin{equation}\label{eq: mean-value constraint for v_i}
    \frac{1}{n} \sum_{i=1}^n v_i = \frac{c}{\sqrt{1 - \ep p}}\,. 
\end{equation}
Minimizing \eqref{eq: equivalent const gl fnl} subject to  the {mean-value constraint} \eqref{eq: mean-value constraint for v_i} leads to the Euler--Lagrange equations 
\begin{equation} \label{eq: const graph EL equation}
    v_i^3 - v_i = \tau \quad \text{and} \quad \frac{1}{n} \sum_{i  =1}^n v_i = \frac{c}{\sqrt{1 - \ep p}} \,, 
\end{equation}
where $\tau$ is a Lagrange multiplier that is associated with the {mean-value constraint}.

\begin{proposition}[Characterization of GL minimizers for the complete graph with constant edge weight]\label{prop: const gl minimizers}
Let $W_n \equiv p$, and suppose that $|c| < \sqrt{1 - \ep p}$ for sufficiently small $\ep$.
The minimizers of $\GL^{W_n}_\ep$ {under the {mean-value constraint} \eqref{eq: mean-value constraint for v_i}} are functions $u$ with 
at most three different values in their codomains. These values are in the set
$\{ z_{\pm, n}\}$, where  $z_{\pm, n} = \pm \sqrt{1 - \ep p} + O(1/\sqrt{n})$,
except for at most one node on which $u$ takes a value that is $O(1/\sqrt{n})$. 
\end{proposition}

As $n \ra \infty$ and $\ep \ra 0$, the values
$\{\pm \sqrt{1 - \ep p} + O(1/\sqrt{n})\}$ of the minimizers approach $\{\pm 1\}$. The minimizers are not unique because the energy \eqref{eq: equivalent const gl fnl} and the constraints \eqref{eq: mean-value constraint for v_i} are invariant with respect to permutations of the subscripts.

We now prove Proposition \ref{prop: const gl minimizers}.

\begin{proof}

With the change of variables \eqref{eq: change of vars}, it suffices to describes minimizers $v$ of the equivalent minimization problem \eqref{eq: equivalent const gl fnl}. 

{We first show} that the energy functional \eqref{eq: equivalent const gl fnl} for a minimizer $v$ is bounded above by $1$. 
Consider a candidate minimizer $\tilde{v}$ with $\tilde{v}_i\in \{\pm 1\}$ for all {$i \in [n]$} that takes the value $+1$ on $\frac{n}{2}(1 + \frac{c}{\sqrt{1 - \ep p}})$ of the nodes and the value $-1$ on $\frac{n}{2}(1 - \frac{c}{\sqrt{1 - \ep p}})$ of the nodes. If these numbers of positive and negative values are not integers, we round the numbers; because the two numbers sum to $n$, one rounds up and one rounds down. (It does not matter which one rounds up and which one rounds down.)
There may be one node $j$ that does not have an assigned value. On node $j$, the candidate minimizer $\tilde{v}$ takes whichever value allows it to satisfy the {mean-value constraint}. Namely, $\tilde{v}_j = \frac{nc}{\sqrt{1 - \ep p}} \, - \, \text{round}(\frac{nc}{\sqrt{1 - \ep p}})$, where ``round" maps a real number to the nearest integer and $0.5$ rounds up. Because $|\tilde{v}_j| < 1$, the energy is $E_n(\tilde{v})\leq 1$ for this choice of $\tilde{v}$. For a minimizer $v$, we have $E_n(v) \leq E_n(\tilde{v}) \leq 1$. 

We now show that the candidate minimizer is an actual minimizer. By equation \eqref{eq: equivalent const gl fnl}, the energy bound $E_n(v) \leq 1$ implies that there must be some {node} $j$ such that  $|v_j^2 - 1|\leq 1/\sqrt{n}$. Because $|\tau| = |v_j^2 - 1| |v_j|\leq (1/\sqrt{n})(1 +  1/\sqrt{n})^{1/2}$, 
we have {$|\tau| = O(1/\sqrt{n})$}.
For large $n$, the cubic polynomial
$g(x) = x^3 - x$ is a monotone function 
around its three simple roots and the derivatives at the roots are constants that are bounded away from $0$ uniformly with respect to $n$. Therefore, whenever $|g(v_i)| = |v_i^2 - 1| |v_i| = |\tau| \leq (1/\sqrt{n})(1 +  1/\sqrt{n})^{1/2}$, {the Inverse Function Theorem guarantees} that either $|v_i^2 - 1| = O(1/\sqrt{n})$ or $|v_i| = O(1/\sqrt{n})$.
{The contribution to $E_n$ of any $v_i$ that satisfies $|v_i^2 - 1| = O(1/\sqrt{n})$ is $O(1/\sqrt{n})$, {and} the contribution to $E_n$ of any $v_i$ that satisfies $|v_i| = O(\sqrt{1/n})$ is at least $(1 - {O(1/n)} )^2$.} Therefore, with the bound $E_n(v) \leq 1$, we see that there is at most one node $i$ with $|v_i| = O(1/\sqrt{n})$ {for} sufficiently large $n$.
Every other value of $v$ {must} be a root 
of $g(x) = \tau$ and
satisfies $|v_i^2 - 1| = O(1/\sqrt{n})$. {Therefore,} $v_i\to \pm 1$ as $n\to \infty$.

\end{proof}

The minimizer $v$ of \eqref{eq: equivalent const gl fnl} that we constructed has a discontinuous profile, as it has jumps between the values $+ \sqrt{1 - \ep p}$ and $-\sqrt{1 - \ep p}$. See the related discussion in \cite{du2016asymptotically}.
Furthermore, as $n \ra \infty$, the limit $v$ of the sequence of minimizers takes the values $\pm 1$ almost everywhere as $\ep \ra 0$, so (by the change of variables \eqref{eq: change of vars}) $u$ takes values $\pm \sqrt{1 - \ep p}$ almost everywhere.  

We now examine where the optimal $u$ takes the values $\pm \sqrt{1 - \ep p}$. In doing so, we illustrate the properties of the GL minimizer in the graphon limit. Because $\GL^{W_n}_\ep \gamra \GL^W_\ep$ as $W_n \sqra W$, we know (see Section \ref{sec: limit (3)}) that characterizing the GL minimizers for the constant graph gives insight into GL minimizers for the constant graphon.
We show that the (arbitrarily many) oscillations between the two values $+\sqrt{1 - \ep p}$ and $-\sqrt{1 - \ep p}$ do not affect the optimality of $u$. The following discussion also helps explain why graphon functionals act on Young measures. We consider two different Young measures that minimize \eqref{eq: equivalent const gl fnl} (and thus minimize the constant-graphon GL functional) and that satisfy the mean-value constraint~\eqref{eq: graphon mean-value constraint} but have different amounts of oscillation.

First, consider the Young measure $\{\nu^{(1)}_x\}_{x\in (0,1)}$ that is defined by
\begin{equation}
\label{eq: const ym minimizer}
    \nu^{(1)}_x =  \frac{1 + c}{2} \delta_{\sqrt{1 - \ep p}} + \frac{1 - c}{2} \delta_{-\sqrt{1 - \ep p}} \,. 
\end{equation}

The Young measure $\nu^{(1)}$ is the limit, with respect to narrow convergence in the space of Young measures, of sequences of increasingly oscillatory functions $u$ that take the value $\sqrt{1 - \ep p}$ on a proportion $\frac{1 + c}{2}$ of the points in $[0,1]$ and the value $-\sqrt{1 - \ep p}$ on a proportion $\frac{1 - c}{2}$ of the points in $[0,1]$. 
Although $\nu^{(1)}$ is not a function, one can think of it as the limit of increasingly oscillatory functions that satisfy the constraint \eqref{eq: graphon mean-value constraint}. 
The graphon GL functional \eqref{eq: graphon GL} evaluated {at} $\nu^{(1)}$ is 
\begin{equation} \label{eq: gl for nu one}
    {\GL^W_\ep(\nu^{(1)}) = 2 {(1 - \ep p)} (1 - c)(1 + c) \int_0^1 \int_0^1 W(x,y) \,dx \, dy   + \ep p^2} \,. 
\end{equation}
The first term on the right-hand side of~\eqref{eq: gl for nu one} is the Dirichlet-energy term, and the second term (i.e., $\ep p^2$) is the double-well term.

Second, consider the Young measure $\{\nu^{(2)}_x\}_{x\in (0,1)}$ that is defined by 
\begin{equation}
	\nu^{(2)}_x = \delta_{u(x)}\,, \, \text{ where } \, u(x) 
		= \begin{cases}
		        -\sqrt{1 - \ep p}\, , &0 < x < \frac{1 - c}{2} \\
        			\sqrt{1 - \ep p}\, , &\frac{1 - c}{2} < x < 1\,,
    \end{cases}
\end{equation}
which represents a non-oscillatory limit because (1) $u$ has a single transition and (2) $u$ has a fixed value on each side of the transition. The location of this transition ensures that $\nu^{(2)}$ satisfies the mean-value} constraint \eqref{eq: graphon mean-value constraint}. The graphon GL functional \eqref{eq: graphon GL} evaluated at $\nu^{(2)}$ is 
    \begin{equation}\label{eq: gl for nu two}
        {\GL^W_\ep(\nu^{(2)}) = 8 (1 - \ep p) \int_0^{\frac{1 - c}{2}} \int_{\frac{1 - c}{2}}^1 W(x,y) \,dx \,dy + \ep p^2} \,. 
    \end{equation}
As in equation \eqref{eq: gl for nu one}, the first term on the right-hand side of~\eqref{eq: gl for nu two} is the Dirichlet-energy term and the second term is the double-well term. 

To illustrate the effect of oscillations on the value of the graphon GL functional {\eqref{eq: graphon GL} for the constant graphon $W \equiv p$}, we compare the values of $\GL^W_\ep(\nu^{(1)})$ and $\GL^W_\ep(\nu^{(2)})$. Consider the case $W \equiv 1$ and suppose that the {mean-value constraint} \eqref{eq: graphon mean-value constraint} holds. Equations \eqref{eq: gl for nu two} and \eqref{eq: gl for nu one} yield
\begin{equation*}
    \GL^W_\ep(\nu^{(1)}) = \GL^W_\ep(\nu^{(2)}) = 2{(1 - \ep p)} (1 - c)(1 + c) + \ep p^2 \,,
\end{equation*}
which illustrates that oscillations do not increase the value of the graphon GL for the constant graphon.
This contrasts with what occurs for classical GL functionals, which are concerned with interfaces and smoothness \cite{du2020phase}.

The contribution to the constant-graphon GL functional $\GL^W_\ep$ is the same whenever $\nu_x$ and $\nu_y$ differ from each other, regardless of how ``far" $x$ and $y$ are from each other. This situation contrasts sharply with classical GL functionals, where the contribution from variations in the value of $u$ is $|\nabla u(x)|^2$ (which penalizes local changes in the value of $u$ near $x$), rather than $W(x,y)|u(x) - u(y)|^2$ (which {penalizes nonlocal} variations in values).

In the $\ep \ra 0$ limit, the Young-measure minimizers \eqref{eq: const ym minimizer} converge to the limiting Young measure $\nu = \{\nu_x\}$, which is defined by
\begin{equation}\label{eq: const tv minimizer}
    \nu_x = \theta \delta_1 + (1 - \theta) \delta_{-1}
\end{equation}
for all $x \in [0,1]$. Because of the $\Gamma$-convergence $\GL^W_\ep \gamra \TV^W$ (see Section \ref{sec: limit (4)}), the Young-measure limit \eqref{eq: const tv minimizer} of the minimizers \eqref{eq: const ym minimizer} of $\GL^W_\ep$ is a minimizer of the limiting energy $\TV^W$ for $W \equiv p$. 

To help with later discussions, we also study the (unrealistic) case of ``oversaturation", which occurs when the {mean-value constraint} imposes that the mean value of $v$ is larger than $1$ or smaller than $-1$. Oversaturation occurs when $|c| > {\sqrt{1 - \ep p}}$. In the following lemma, we characterize the minimizers when $|c| \geq {\sqrt{1 - \ep p}}$ and show that the optimal value of $|c|$ is $\sqrt{1 - \ep p}$ when there is oversaturation.

\begin{lemma}[{Non-oversaturation of} GL minimizers]\label{lem: oversaturation const case}
When $|\gamma| := \left|\frac{c}{\sqrt{1 - \ep p}}\right| \geq 1$ in the {mean-value constraint} \eqref{eq: mean-value constraint for v_i}, the minimizer of \eqref{eq: equivalent const gl fnl} is the constant function $v^*$, which satisfies $v^*_i = \gamma$ for all $i \in [n]$; the minimum energy is $E_n(v^*) = n(\gamma^2 - 1)^2$. Equivalently, the minimizer of {the constant-graph GL functional} $\GL^{W_n}_\ep (u)$ with the {mean-value constraint} \eqref{eq: graph mean-value constraint} is $u^*$ (which satisfies $u_i^* = c$ for all $i \in [n]$), and the minimum energy is $\GL^{W_n}_\ep (u^*) = \frac{1}{\epsilon}\Phi(c) = \frac{1}{\epsilon}(c^2 - 1)^2$. Additionally, the energies $E_n(v^*)$ and $\GL^{W_n}_\ep (u^*)$ have smaller minima when $|c| = \sqrt{1-\ep p}$ than when $|c| > \sqrt{1 - \ep p}$.
\end{lemma}

\begin{proof}
    Let $\gamma = \frac{c}{\sqrt{1 - \ep p}}$, and rewrite the {mean-value constraint} \eqref{eq: graph mean-value constraint} as 
    \begin{equation}\label{eq: equiv graph vol constraint}
        \sum_{i = 1}^n v_i = \gamma n \,. 
    \end{equation}    
    
    We begin by showing that the minimizer of $E_n$ is
    the constant function. Let $h$ be the convex hull of $f(s) = (s^2 - 1)^2${. Therefore,} $h = 0$ on $[-1,1]$ and $h = f$ everywhere else.

   Define 
    \begin{equation*}
        H_n(v) := \sum_{i = 1}^n h(v_i) \,,
    \end{equation*}
which satisfies $H_n(v) \leq E_n(v)$ for any $v$.
Because $h$ is convex, it follows that $(h(x) + h(y))/2 \leq h\left((x + y)/2\right)$. Therefore, if $v_i \neq v_j$ for any $i \neq j$, the function $H_n$ decreases if $v_i$ and $v_j$ are both replaced by their mean $(v_i + v_j)/2$. Consequently, with the {mean-value constraint} \eqref{eq: equiv graph vol constraint}, the minimizer of $H_n$ is the constant function $v' \equiv \gamma = \frac{c}{\sqrt{1 - \ep p}}$. 
    Because $\gamma\geq 1$, we have $H_n(v') = E_n(v')$, so the minimizer of $E_n$ entails that $v_i = \gamma$, which yields 
    \begin{equation*}
    	E_n(v') = H_n(v') = n \left(\gamma^2 - 1\right)^2\,. 
    \end{equation*}
    The conclusion about $\GL^{W_n}_\ep(u)$ then follows immediately.
 \end{proof}


\subsection{\texorpdfstring{$2 \times 2$ stochastic block models (SBMs)}{2 x 2 stochastic block models (SBMs)}} \label{sec: sbm}

Consider the $2 \times 2$ piecewise-constant graphon
\begin{equation} \label{eq: sbm}
    W(x,y) = 
    \begin{cases}
        w_{11} &\text{ if } \,\, x \in S\,,\, y\in S \\
        w_{12} &\text{ if } \,\, x \in S\,,\, y\in S^c \\
        w_{21} &\text{ if } \,\, x \in S^c\,,\, y\in S \\
        w_{22} &\text{ if } \,\, x\in S^c\,,\, y\in S^c  \,,
    \end{cases}
\end{equation} 
where $w_{k \ell} \in [0,1]$ are edge weights and $S = (0, a)$ and $S^c = [a, 1)$, for some $a\in (0,1)$, are communities in $(0,1)$. This graphon is a stochastic block model (SBM)~\cite{newman2018}. Researchers use SBMs as generative models of graphs with various types of mesoscale network structures, such as assortative or disassortative block structures~\cite{peixoto2019}. SBMs are a relatively general class of graphs and graphons. Indeed, the Szemeredi Lemma implies that one can approximate any {$L^\infty$} graphon arbitrarily closely {in {cut norm} by an SBM~\cite[Lemma 3.1]{lovasz2007szemeredi}.}

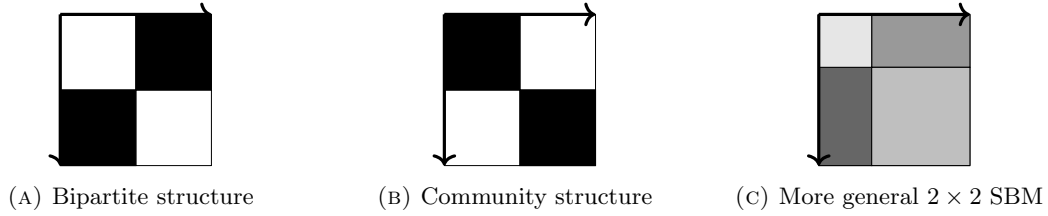
\begin{figure}[h!]
\captionsetup[subfigure]{font=footnotesize}
\centering
    \subcaptionbox{Bipartite {structure}}[.3\textwidth]{%
        \begin{tikzpicture}
            \draw[help lines, color=black, solid] (0,0) grid (2,2);
            \draw[->,very thick] (0,2)--(0,0);
            \draw[->,very thick] (0,2)--(2,2);
            \draw[color=black, fill = black] (0,0) rectangle (1,1);
            \draw[color=black, fill = black] (1,1) rectangle (2,2);
        \end{tikzpicture}
    }
    \subcaptionbox{Community structure}[.3\textwidth]{%
         \begin{tikzpicture}
                \draw[help lines, color=black, solid] (0,0) grid (2,2);
                \draw[->,very thick] (0,2)--(0,0);
                \draw[->,very thick] (0,2)--(2,2);
                
                \draw[color=black, fill = black] (0,1) rectangle (1,2);
                \draw[color=black, fill = black] (1,0) rectangle (2,1);
            \end{tikzpicture}
    }%
    \subcaptionbox{More general $2 \times 2$ SBM}[.3\textwidth]{%
        \begin{tikzpicture}
            \draw[help lines, color=black, solid] (0,0) grid (2,2);
            
            \draw[color=black, fill = gray!120] (0,0) rectangle (.7,1.3);
            \draw[color=black, fill = gray!80] (.7,1.3) rectangle (2,2);
            \draw[color=black, fill = gray!20] (0,1.3) rectangle (.7,2);
            \draw[color=black, fill = gray!50] (.7,0) rectangle (2,1.3);
            
            \draw[->,very thick] (0,2)--(0,0);
            \draw[->,very thick] (0,2)--(2,2);
        \end{tikzpicture}
    }
\caption{Three types of $2\times 2$ piecewise-constant graphons. 
}
\label{fig: 2x2}
\end{figure}

{Consider} the $2 \times 2$ piecewise-constant graphons in Figure \ref{fig: 2x2}. If one views {an} SBM as a generative model, the entries {$w_{k \ell}$} encode the probability that there is an edge between a node in community {$k \in \{1,2\}$} and a node in community $\ell \in \{1,2\}$. However, we view the entries $w_{k \ell}$ as fixed values that form a block-diagonal adjacency matrix in which adjacency-matrix entries take the values 
of $W(x,y)$ in \eqref{eq: sbm}.

In an assortative SBM, the intra-community edge probabilities $w_{kk}$ are larger than the inter-community edge probabilities $w_{k \ell}$ (with $k \neq \ell$).
In Figure \ref{fig: 2x2}, we show three examples of graphons; in these examples, black blocks have the value $1$, white blocks have the value {$0$}, and gray blocks have values in the interval $(0,1)$. The origin is at the upper-left corner of each graphon. An SBM with planted community structure (see Figure \ref{fig: 2x2}(b)) is an extreme case of {an} assortative block structure~\cite{peixoto2019}. In a disassortative SBM, the inter-community edge probabilities $w_{k \ell}$ (with $k \neq \ell$) are larger than the intra-community edge probabiliites $w_{kk}$, so inter-community edges are more likely than intra-community edges. (In this case, the term ``community" is technically a misnomer, but we use it for simplicity.)
An extreme example is a bipartite SBM (see Figure \ref{fig: 2x2}(a)). 
Finally, Figure \ref{fig: 2x2}(c) shows an example of a general $2 \times 2$ SBM with two distinct community sizes. When {$w_{21} = w_{12} = 1$ and $w_{11} = w_{22} = 0$} in equation~\eqref{eq: sbm}, the graphon $W$ is a  bipartite graphon. When 
{$w_{21} = w_{12} = 0$ and $w_{11} = w_{22} = 1$}, the graphon $W$ is a community-structure graphon with two weakly-connected subgraphons. 

Seeking a GL minimizer for SBMs is different from seeking a constant-graphon GL minimizer because {we do not need a {mean-value constraint} 
in the SBM case. For the constant graphon, we need a {mean-value constraint} so that the minimizer is not trivial (i.e., all $1$ or all $-1$), but a bipartite graphon does not possess such a trivial minimizer. Trivial minimizers are prevented by the block structure that is encoded in the GL functional, rather than by an explicit constraint. Nontrivial block structures (e.g., ones that correpond to bipartite structure or community structure) incentivize a GL minimizer to distribute positive and negative values to different sets of nodes.


\subsubsection{Complete bipartite graphon} \label{sec: 2x2 bipartite}

In a bipartite graph (see Figure \ref{fig: 2x2}(b)), each node is in one of two sets, $S$ and $S^c$, with edges only between nodes in different sets. Consider the complete bipartite graph 
\begin{equation}\label{eq: bipartite graph}
    A^{(n)}_{ij} = \begin{cases}
        1 &\text{ if }\,\, i\in S\,, \, j\in S^c 
        \\
        1 &\text{ if } \,\, i\in S^c\,, \, j \in S
        \\
        0 &\text{ if } \,\, i\in S\, , \, j\in S
        \\
        0 &\text{ if } \,\, i \in S\,, \, j \in S^c \,. 
    \end{cases}
\end{equation}
Suppose that the graph has $n$ nodes.
Let {$a = {|S|}/{n}$ and $1 - a = {|S^c|}/{n}$,} so $a \in (0,1)$ is the proportion of nodes in {the} set $S$. The corresponding complete bipartite graphon is
\begin{equation}\label{eq: bipartite graphon}
	W(x,y) = 
		\begin{cases}
	0 &\text{ if } \,\, (x,y)\in (0,a) \times (0,a) \\
	1 &\text{ if  } \,\, (x,y)\in (0,a)\times [a,1) \\
	1 &\text{ if  } \,\, (x,y)\in [a,1) \times (0,a) \\
	0 &\text{ if  } \,\, (x,y)\in [a,1) \times [a,1)\,.
		\end{cases} 
\end{equation}

We rewrite the GL functional for the bipartite graph as 
\begin{align}\label{above}
    \GL^{W_n}_\ep(u) &= \frac{1}{n^2} \sum_{\substack{i\in S  , \, j\in S^c}} (u_i - u_j)^2 + \frac{1}{\ep n} \sum_{i=1}^n \Phi(u_i)   \\
     		&= \frac{1}{n^2} \sum_{\substack{i\in S , \, j\in S^c}} u_i^2
     			+ \frac{1}{n^2} \sum_{\substack{i\in S  , \, j\in S^c}} u_j^2 - \frac{2}{n^2} \sum_{i\in S , \, j\in S^c} u_j
    			  + \frac{1}{\ep n} \sum_{i=1}^n \Phi(u_i)  \notag  \\
           &= \frac{1}{n^2} \sum_{i\in S}  |S^c| u_i^2 + \frac{1}{n^2} \sum_{ j\in S^c} |S| u_j^2 - \frac{2}{n^2} \sum_{i\in S , \, j\in S^c} u_i u_j
     		 + \frac{1}{\ep n} \sum_{i=1}^n \Phi(u_i)   \notag \\
    &= \frac{1}{\ep n} \sum_{i\in S} \left( u_i^4 - (2 - \ep (1 - a)) u_i^2 + 1\right) 
    		+ \frac{1}{\ep n} \sum_{i\in S^c} \left( u_i^4 - (2- \ep a) u_i^2 + 1\right)
    		- \frac{2}{n^2} \sum_{i\in S} u_i \sum_{j\in S^c}u_j\,. \notag
\end{align}

By completing the square in the first two sums, we see that \eqref{above} is equal (up to a constant) to
\begin{equation}\label{eq: equivalent min problem bipartite}
    \frac{1}{\ep n} \sum_{i\in S} \left(u_i^2 - \left(1 - \frac{\ep(1-a)}{2}\right) \right)^2 + \frac{1}{\ep n} \sum_{i\in S^c} \left(u_i^2 - \left(1- \frac{\ep a}{2}\right) \right)^2 - \frac{2}{n^2} \sum_{i\in S} u_i \sum_{j\in S^c}u_j\,.
\end{equation}
Define $\alpha = \sum_{i\in S} u_i$, which we combine with the mean-value constraint \eqref{eq: graph mean-value constraint} and insert into the last term of \eqref{eq: equivalent min problem bipartite} to obtain
\begin{equation*} 
	{\sum_{i\in S} u_i \sum_{j\in S^c} u_j} 
    = \alpha(c - \alpha) \,.
\end{equation*} 
Note that we are treating the mean-value constraint \eqref{eq: graph mean-value constraint} as a shorthand notation for the mean value of $u_i$, rather than as a genuine constraint. 

We introduce the rescaling
\begin{equation}\label{eq: bipartite change of vars}
    v_i = \begin{cases}
        \frac{u_i}{\sqrt{1 - \ep(1 - a)/2}} &\text{ for } i\in S 
        \\
        \frac{u_i}{\sqrt{1 - \ep a/2}} &\text{ for } i\in S^c \,,
    \end{cases}
\end{equation}
which is similar to the rescaling \eqref{eq: change of vars}. To simplify notation, we define 
\begin{equation}  
    c_S := \sqrt{1 - \ep(1 - a)/2} \;\; \text{ and } \;\; c_{S^c} := \sqrt{1 - \ep a/2} \,.
\end{equation}
This yields the equivalent minimization problem 
\begin{equation}\label{eq: euivalent bipartite fnl}
    E_n(v) :=  \frac{c_S^4}{\ep} \sum_{i\in S} (v_i^2 - 1)^2 + \frac{c_{S^c}^4}{\ep} \sum_{i\in S^c} (v_i^2 - 1)^2 - \frac{2}{n} c_S c_{S^c}  \alpha(c - \alpha) \,.  
\end{equation}
With the change of variables \eqref{eq: bipartite change of vars}, we then obtain
\begin{equation}\label{eq: bipartite vol constraint}
    \alpha = c_S \sum_{i\in S} v_i \,, \quad  c - \alpha = c_{S^c} \sum_{i\in S^c} v_i \,.
\end{equation}
Treating $v$ and $\alpha$ as unknowns, the resulting {Euler--Lagrange} equations for minimizing $E_n$ are 
\begin{align*}
      \frac{4c_S^4}{\ep} (v_i^3 - v_i ) - \tau_{S}c_S  &= 0 \quad \text{ on } S \,,  \\
    \frac{4c_{S^c}^4}{\ep} (v_i^3 - v_i ) - \tau_{S^c}c_{S^c} &= 0 \quad \text{ on } S^c \,, \\
    -\frac{2}{n} c_S c_{S^c} \, \alpha + \tau_{S} - \tau_{S^c} &= 0 \,. \numberthis \label{eq: bipartite el}
\end{align*}
The first two equations in \eqref{eq: bipartite el} are for minimizing 
$v$ on $S$ and on $S^c$, respectively, and the third equation is for minimizing $\alpha$. There is no constraint equation.

\begin{proposition}[Characterization of the GL minimizers for the complete bipartite graphon]\label{seven-two}
    Let $W_n$ be the complete bipartite graph with $n$ nodes (see \eqref{eq: bipartite graph}). The minimizers of $\GL^{W_n}_\ep$ are functions $u$ that take the values $\{ \pm \sqrt{1 - \ep a/2} + O(1/\sqrt{n})\}$ except for at most a finite number {$m$} of nodes, where $m$ is independent of $n$. 
\end{proposition}

\begin{proof}

We use a similar argument as we did for the constant graphon to show that the values $v_i$ of the global minimizer are close to $\pm 1$ except for at most $m$ {nodes,} where $m$ is independent of $n$.

We show that the minimum value of $E_n$ is bounded above by $3/\ep$ for all $n$. To do this, we define
a function $\tilde{v}$ such that 
\begin{equation}\label{eq: uniform energy bound}
    E_n(\tilde{v}) \leq \frac{3}{\ep} \,.
\end{equation}
We need to consider different parities of $n$, $|S|$, and $|S^c|$.
We first suppose that $n$, $|S|$, and $|S^c|$ are all even.
In this case, $E_n(\tilde{v}) =  0$ if $\tilde{v}$ has the same number of $1$ values on $S$ as the number of $-1$ values on $S^c$. We thus obtain $\alpha = 0$. xNow suppose that $n$ is even but both $|S|$ and $|S^c|$ are odd, we set $\tilde{v}_i = 0$ for one node in $S$ and one node in $S^c$, and we {let $\tilde{v}$ attain the values $1$ and $-1$ on an equal number of entries} in $S$ {and} $S^c$. {We then again obtain} $\alpha = 0$. For both of these choices of $\tilde{v}$, we have 
\begin{equation}\label{Eq: E_n bound}
    E_n(\tilde{v}) \leq \frac{c_S^4 + c_{S^c}^4}{\ep} = \frac{2}{\ep} - 1 + \frac{\ep}{4}(1 - 2a + 2a^2) \leq \frac{3}{\ep} 
\end{equation}
for $a\in (0,1)$ and $\ep < 1$. 
{In} the final case, $n$ is odd, so one of $|S|$ and $|S^c|$ is even and the other is odd. In this case, one node has $\tilde{v}_i = 0$ and the remaining $\tilde{v}_i$ have an equal number of $1$ and $-1$ entries in $S$ and an equal number of $1$ and $-1$ entries in $S^c$. This again yields $\alpha = 0$, and $E_n(\tilde{v})$ is either {${c_S^4}/{\ep}$} (if $|S|$ is odd) or {${c_{S^c}^4}/{\ep}$} (if $|S^c|$ is odd), which again satisfies the uniform energy bound \eqref{eq: uniform energy bound}.

From the energy upper bound \eqref{eq: uniform energy bound}, a minimizer $v$ satisfies
\begin{equation}
\label{Eq: Energy-min}
     \frac{c_S^4}{\ep} \sum_{i\in S} (v_i^2 - 1)^2
    + \frac{c_{S^c}^4}{\ep} \sum_{i\in S^c} (v_i^2 - 1)^2 \leq E_n(v)\leq E_n(\tilde{v}) \leq \frac{3}{\ep} \,.
\end{equation}
Therefore, there exist nodes $j\in S$
and $j'\in S^c$ such that 
\begin{equation}
    (v_j^2 - 1)^2 \leq \frac{3}{|S|c^4_S} \quad \text{and} \quad (v_{j'}^2 - 1)^2 \leq \frac{3}{|S^c|c^4_S}\,.
\end{equation}

Because $|S| = an$ and $|S^c| = (1 - a)n$, we have 
\begin{equation*}
    |v_j^2 - 1| \leq \frac{m_1}{\sqrt{n}} \quad \text{ and } \quad {|v_{j'}^2 - 1|} \leq \frac{m_2}{\sqrt{n}} \,,
\end{equation*}
where the constants $m_1$ and $m_2$ do not depend on $n$. {The Euler--Lagrange} equations \eqref{eq: bipartite el} hold for all nodes, and they thus hold for nodes $j$ and $j'$. Therefore,
\begin{equation*}
    | \tau_S  | \leq \frac{C}{\sqrt{n}} \quad \text{ and }  \quad  |\tau_{S^c}| \leq \frac{C}{\sqrt{n}} 
\end{equation*}
for some constant $C$ that does not depend on $n$. With these bounds on $\tau_S$ and $\tau_{S^c}$, equations \eqref{eq: bipartite el} yield
\begin{align*}
    v_i^3 - v_i = O(1/\sqrt{n}) \,
\end{align*}
for all nodes in both $S$ and $S^c$. Similarly to the proof of Proposition \ref{prop: const gl minimizers}, for sufficiently large $n$, the zeros $v_i$ lie within $O(k/\sqrt{n})$ of $\pm 1$ and $0$. Furthermore, due to the energy bound \eqref{Eq: Energy-min}, we see that $v$ can be within $O(k/\sqrt{n})$ of $0$ for at most 
a finite number of nodes. {This number of nodes} does not exceed a constant that is independent of $n$.

From the change of variables \eqref{eq: bipartite change of vars}, the GL minimizer $u$ takes the values $u_i =  \pm \sqrt{1 - \ep a/2} \\ + O(1/\sqrt{n})$ for $i \in S^c$ and $u_i = \pm \sqrt{1 - \ep(1 - a)/2} + O(1/\sqrt{n}) $ for $i \in S$, except for at most a finite number of nodes $i$. 
\end{proof}

It follows from Proposition \ref{seven-two} that the limiting Young measures for the {complete bipartite} graphon GL minimizer are
\begin{align*}
    \nu_x &= \pmb{1}_{U_1}(x) \, \delta_{\sqrt{1 - \ep a/2}} + \pmb{1}_{S \setminus U_1}(x) \, \delta_{-\sqrt{1 - \ep a/2}}
    \\
    &\quad +  \pmb{1}_{U_2}(x) \, \delta_{\sqrt{1 - \ep(1 - a)/2}} + \pmb{1}_{S^c \setminus U_2}(x) \, \delta_{-\sqrt{1 - \ep(1 - a)/2}} \,,
\end{align*}
where $U_1 \subset S$ and $U_2 \subset S^c$ are any sets of sizes
$|U_1| = {|S|}/{2}$ and $|U_2| = {|S^c|}/{2}$.

As $n \ra \infty$, the energy bound \eqref{Eq: Energy-min} also yields a bound on the limiting graphon energy. {This bound on the limiting graphon energy shrinks as $n$ increases.} Because one loses the discrete character of graphs in the continuum limit, it no longer makes sense to discuss ``even" or ``odd" $|S|$ and $|S^c|$.


\subsubsection{Community-structure graphon}

Community-structure graphs consist of densely-connected subgraphs (i.e., communities) that are sparsely connected to each other. As in a $2 \times 2$ bipartite SBM, a $2 \times 2$ community-structure graph involves a partition $\{S,S^c\}$ of the set of nodes of a graph into two communities. Edges occur frequently between nodes in the same community (either $S$ or $S^c$), and they occur sparsely between nodes in different communities. We suppose that the communities are complete, and we refer to this example as the ``complete community-structure graph". The complete community-structure graph with communities $S$ and $S^c$ has adjacency-matrix elements 
\begin{equation}\label{eq: community graph}
    A^{(n)}_{ij} = \begin{cases}
        1 &\text{ if }\,\, i\,, \,j \in S
        \\
        1 &\text{ if } \,\,i\,, \,j \in S^c
        \\
        0 &\text{ if } \,\,i \in S\,,\,  j\in S^c
        \\
        0 &\text{ if } \,\,i \in S^c\,,\,  j \in S \,. 
    \end{cases}
\end{equation}
The corresponding complete community-structure graphon is
\begin{equation}\label{eq: community graphon}
	W(x,y) = 
		\begin{cases}
	1 &\text{ if } \,\, (x,y)\in (0,a) \times (0,a) \\
	0 &\text{ if  } \,\, (x,y)\in (0,a)\times [a,1) \\
	0 &\text{ if  } \,\, (x,y)\in [a,1) \times (0,a) \\
	1 &\text{ if  } \,\, (x,y)\in [a,1) \times [a,1)\,.
		\end{cases} 
\end{equation}

The GL functional for a complete community-structure graph is 
\begin{align*}
    \GL^{W_n}_\ep(u) &= \frac{1}{n^2} \sum_{\substack{i\in S , \, j\in S}} (u_i - u_j)^2 + \frac{1}{n^2} \sum_{\substack{i\in S^c , \, j\in S^c}} (u_i - u_j)^2 
    + \frac{1}{\ep n} \sum_{i = 1}^n \Phi(u_i) \,,
\end{align*}
which, with a similar computation to that for a complete bipartite graph, is equivalent up to a constant to the functional
\begin{equation}\label{eq: equivalent community fnl}
    E_n(v) =  \frac{c_S^4}{\ep} \sum_{i\in S} (v_i^2 - 1)^2 + \frac{c_{S^c}^4}{\ep} \sum_{i\in S^c} (v_i^2 - 1)^2 - \frac{2}{n} c_S c_{S^c} \alpha^2 - \frac{2}{n}c_S c_{S^c}(c - \alpha)^2 \,,
\end{equation}
where $\alpha$ and $c - \alpha$ are the same (see \eqref{eq: bipartite vol constraint}) as for a complete bipartite graphon, $c_S v_i = u_i$ for $i\in S$ and $c_{S^c} v_i = u_i$ for $i\in S^c$, and 
\begin{equation}\label{eq: communities scaling}
    c_S = \sqrt{1 - \ep a} \,, \quad  c_{S^c} = \sqrt{1 - \ep(1 - a)} \,. 
\end{equation}
As for a complete bipartite graphon, $\alpha$ does not act as a constraint. Instead, it provides a concise notation for the quantity $\alpha = \sum_{i\in S} u_i$.} Let $c = 0$. 
For $\alpha = \sum_{i\in S} u_i$, we have 
\begin{equation}\label{eq: vol constraint community}
    c_S \sum_{i\in S} v_i = \alpha \,, \quad c_{S^c} \sum_{i\in S^c} v_i = -\alpha \,.
\end{equation}

When $|S| = |S^c|$, the energy functional \eqref{eq: equivalent community fnl} reduces to 
\begin{equation}\label{eq: equivalent community fnl equal communities}
    E_n(v) =  \frac{2c_S^4}{\ep} \sum_{i = 1}^n (v_i^2 - 1)^2 - \frac{4}{n} c_S^2 \alpha^2 \,.
\end{equation}
The {first} term of the energy \eqref{eq: equivalent community fnl equal communities} encourages the values of $v$ to be near $\pm 1$. The {second} term of \eqref{eq: equivalent community fnl equal communities} encourages {$|\alpha|$} to be as large as possible. In other words, the second term encourages the values of $v_i$ within a community to either all be very large or all be very small. Therefore, there is a tradeoff between the first and the second terms.
The {mean-value constraint} \eqref{eq: vol constraint community} ensures that the sum of the values of $v$ in $S$ is the negative of the sum of the values of $v$ in $S^c$. We show {in Proposition \ref{prop: community structure minimizers}} that the optimal balance in this tradeoff has $v$ near $\pm 1$ and $\alpha = |S|$. This implies that $v_i = +1$ for all $i \in S$ and $v_i = -1$ for all $i \in S^c$.

In the following proposition, we characterize the GL minimizers {for} a complete community-structure graph when $|S| = |S^c|$. 
Let $|S| = |S^c|$ (which implies that $c_S = c_{S^c}$) {and $\alpha = \gamma |S|$, where we recall that $\gamma = \frac{c}{\sqrt{1 - \ep p}}$ (see Lemma \ref{lem: oversaturation const case}).}
The {mean-value constraint} \eqref{eq: vol constraint community} then entails that ${\gamma}/{c_S}$ is the mean value of $v$ on $S$ and that $-{\gamma}/{c_S}$ is the mean value of $v$ on $S^c$. Therefore, we can rewrite \eqref{eq: vol constraint community} as 
\begin{equation}\label{eq: equal community vol constraint}
    \frac{1}{|S|} \sum_{i\in S} v_i = \frac{\gamma}{c_S}\,, \quad  -\frac{1}{|S|} \sum_{i\in S^c} v_i = -\frac{\gamma}{c_S}\,.
\end{equation}

\begin{proposition}[Characterization of the GL minimizers for the complete community-structure graph] \label{prop: community structure minimizers}
Let $W_n$ be the complete community-structure graph, which has adjacency-matrix elements \eqref{eq: community graph}, and suppose that $|S| = |S^c|$. 
The minimizers of $\GL^{W_n}_\ep$ are functions $u$ that, on $S$, take a constant value that approaches $+1$ as $\ep \downarrow 0$ and, on $S^c$, take a constant value that approaches $-1$ as $\ep \downarrow 0$. Furthermore, the values of $\gamma$ that minimize $E_n$ {are equal to} $\pm c_S$, which approach $\pm 1$ as $\ep \ra 0$.
\end{proposition}

\begin{proof}

We separately consider the two cases $\left|{\gamma}/{c_S}\right| > 1$ and $\left|{\gamma}/{c_S}\right| \leq 1$. 

First suppose that $\left|\frac{\gamma}{c_S}\right| > 1$. 
Let $h$ be the convex hull of the double-well potential $f(s) = (s^2 - 1)^2$. Therefore, $h = 0$ on $[-1,1]$ and $h = f$ everywhere else. Define 
\begin{equation*}
    H_n(v) := \frac{2 c_S^4}{\ep} \sum_{i = 1}^n h(v_i) - \frac{4}{n} c_S^2 \alpha^2\,,
\end{equation*}
which is a lower bound of the energy $E_n$ for all $v$.
Because $h$ is convex, $(h(x) + h(y))/2 \leq h\left((x + y)/2\right)$. Therefore, if $v_i \neq v_j$ for any $i \neq j$, the function $H_n$ decreases if one replaces both $v_i$ and $v_j$ by their mean $(v_i + v_j)/2$.
Consequently, the minimizer $v'$ of $H_n$ is constant on $S$ and constant on $S^c$. However, because of the {mean-value constraint} \eqref{eq: vol constraint community}, $v'$ has different values on $S$ and $S^c$.
The minimizer that satisfies the {mean-value constraint} \eqref{eq: equal community vol constraint} is the function $v'$ with
\begin{equation}\label{eq: v'}
    v'_i = \begin{cases}
        +\frac{\gamma}{c_S} &\text{ if } \,\, i\in S 
        \\
        -\frac{\gamma}{c_S} &\text{ if } \,\, i\in S^c   \,.  
    \end{cases}
\end{equation}

Because $\left|{\gamma}/{c_S}\right| > 1$, we have that $H_n(v') = E_n(v')$. Let $\gamma' = \gamma^2$ and consider the energy
\begin{equation}\label{eq: g_1}
    E_n(v') = \frac{n c_S^4}{\ep} \left(\frac{ \gamma'}{c_S^2} - 1\right)^2 - \frac{4}{n} c_S^2 \gamma' |S|^2
\end{equation}
as a function of $\gamma'$. We denote this function by $g_1(\gamma')$.
Its derivative is
\begin{align}\label{eq: g_1 derivative}
    \frac{d}{d\gamma'}g_1(\gamma') &=  \frac{2c_S^2 n}{\ep} \left(\frac{\gamma'}{c_S^2} - 1\right) - \frac{4}{n} c_S^2 |S|^2\,.
\end{align}
Setting {the right-hand side of} \eqref{eq: g_1 derivative} to $0$ yields the critical point
\begin{equation*}
    \gamma' = c_S^2 \left(1 + \ep a \right) \,.
\end{equation*} 
Therefore, when $|S| = |S^c|$, the {critical values of $\gamma$ are}
\begin{equation*}
    \gamma \in \left \{\pm c_S \sqrt{1 + \ep/2}\right\} = \left\{\pm \sqrt{1 - (\ep/2)^2}\right\}\,. 
\end{equation*} 
The second derivative $\frac{d^2}{d\gamma'^2}g_1(\gamma') = \frac{4|S|}{\ep} > 0$ everywhere, so the critical points are minima of $g_1$. The minimizer \eqref{eq: v'} of $H_n$ is thus
\begin{equation}
    v'_i = \begin{cases}
        +\sqrt{1+\ep/2} &\text{ if }\,\, i\in S 
        \\
        -\sqrt{1+\ep/2} &\text{ if \,\,} i\in S^c   \,.  
    \end{cases}
\end{equation}
Using the change of variables \eqref{eq: communities scaling}, we obtain the GL minimizer $u$, which is defined by 
\begin{equation}
    u_i = \begin{cases}
        +\sqrt{1 - (\ep/2)^2} &\text{ if }\,\, i\in S 
        \\
        -\sqrt{1 - (\ep/2)^2} &\text{ if }\,\, i\in S^c   \,. 
    \end{cases}
\end{equation}
Because $\ep \downarrow 0$, we have that $+\sqrt{1 - (\ep/2)^2} \, \downarrow \, 1$ and $-\sqrt{1 - (\ep/2)^2} \, \uparrow \, -1$. {That is, the positive values of $u$ converge downward to $1$ and the negative values {of $u$} converge upward to $-1$.} Therefore, even when the mean value of $v$ is forced by the {mean-value constraint} \eqref{eq: equal community vol constraint} to have an absolute value larger than $1$, the optimal $u$ approaches $\pm 1$.

Now suppose that $\left|{\gamma}/{c_S}\right| \leq 1$. We construct a candidate minimizer $\tilde{v}$ that satisfies $\gamma = \pm c_S$, and we show that the energy is larger for {$|\gamma/c_S| < 1$ than for $|\gamma/c_S| = 1$.} 
Let $\tilde{v}_i = +1$ for all nodes $i \in S$, and let $\tilde{v}_i = -1$ for all nodes $i \in S^c$. This yields $\gamma = c_S$ and an energy of
\begin{equation} \label{eq: energy bound community}
    E_n(\tilde{v}) = - \frac{4}{n} c_S^2 |S|^2 \,.
\end{equation}
Swapping the signs of $\tilde{v}$ for $S$ and $S^c$ yields $\gamma = -c_S$ and the {same energy \eqref{eq: energy bound community}}. 

Now consider $v$ with $\gamma \in (0, c_S)$, which implies that $|\gamma/c_S| < 1$. {In this case, the energy \eqref{eq: equivalent community fnl} is}
\begin{equation}\label{unsaturated energy}
    E_n(v) =  \frac{2c_S^4}{\ep} \sum_{i = 1}^n (v_i^2 - 1)^2  - \frac{4}{n} c_S^2 \gamma^2 |S|^2 \,.
\end{equation} 
The second term of \eqref{unsaturated energy} is larger than $E_n(\tilde{v})$ and the first term of \eqref{unsaturated energy} is nonnegative, so $E_n(v) \geq E_n(\tilde{v})$. 
We can use the same argument for the energy $E_n(v)$ for $\gamma \in (-c_S, 0)$. 

We conclude that ${\gamma}/{c_S} = \pm 1$ is the optimal mean value of $v$ on $S$ and $S^c$ when $\left|{\gamma}/{c_S}\right| \leq 1$. Furthermore, the candidate minimizer $\tilde{v}$ is {a} minimizer for this value of $\gamma$. 
The GL minimizer $u$ {that corresponds to this candidate minimizer $\tilde{v}$} has components
\begin{equation}
    u_i = \begin{cases}
        +\frac{1}{\sqrt{1 - \ep/2}} &\text{ if }\,\, i\in S 
        \\
        -\frac{1}{\sqrt{1 - \ep/2}} &\text{ if } \,\, i\in S^c \,. 
    \end{cases}
\end{equation}
{As $\ep \downarrow 0$, the values of the {GL} minimizer $u$ approach $\pm 1$.}

\end{proof}


\section{Conclusions and discussion} \label{sec: conclusion}

We studied large-graph limits of Ginzburg--Landau (GL) functionals by taking a functional-analytic view of graphs as nonlocal kernels. We defined graphon GL and {total-variation (TV)} functionals, and we showed that their minimizers are consistent with minimizers of {associated} graph GL and TV functionals. Given a sequence $\{W_n\}$ of graphs that converges in cut norm to a limiting graphon $W$ (i.e., $W_n \sqra W$), the sequence of graph GL functionals~\eqref{eq: graph GL} $\Gamma$-converges to the graphon GL functional~\eqref{eq: graphon GL}, implying that the sequence of minimizers of the GL functionals that correspond to the graphs $W_n$ converges to a minimizer of the GL functional that corresponds to the graphon $W$ as $W_n \sqra W$. We also proved that the sequence of graph TV functionals~\eqref{eq: graph TV} 
$\Gamma$-converges to the graphon TV functional~\eqref{eq: graphon TV} {as $W_n \sqra W$}. Additionally, we showed that the graphon GL and TV functionals satisfy the classical $\Gamma$-limits $\GL_\ep \gamra \TV$ and $\GL^{W_n}_\ep \gamra \TV^{W_n}$ as $\ep \ra 0$. These $\Gamma$-convergence results (see our summary in Figure~\ref{figure:convergences}), in concert with compactness properties, imply that the minimizers of the $\Gamma$-converging functionals also converge. 

The limiting functionals highlight several fundamental differences between the graphon GL functional and both the graph GL functional and the classical GL functional. One difference is that the graphon functionals are formulated using Young measures, rather than using functions. This difference highlights the fact that the limiting minimizers are Young measures, which can have arbitrary amounts of oscillation, that take the same values in the same proportions as the minimizing functions. Another difference, which is highlighted by our limit (4) (see Section \ref{sec: limit (4)} for the proof of this limit), is that the $\ep$-scaling of the graphon GL functional is somewhere between the $\ep$-scalings of the classical GL functional and the graph GL functional. 
The classical GL functional has scalings $\ep$ and $1/\ep$ for the Dirichlet energy and double-well potential, respectively. By contrast, the graph GL functional has the scalings $1$ and $1/\ep$, respectively. The graphon GL functional has the same $\ep$-scalings as the graph GL functional when the graphon $W$ is bounded. We did not determine a scaling for the more general $W \in L^p((0,1)^2)$, but this is worth examining in future efforts.

Our limit (3), which we proved in Section \ref{sec: limit (3)}, extends results of Braides et al.~\cite[Lemma 11, Theorem 12]{braides2020cut} in two key ways. First, Braides et al. proved $\Gamma$-convergence of the graph-cut functional (which acts on finite-range functions) as $n \ra \infty$, whereas we proved $\Gamma$-convergence of the graph Dirichlet energy, which extends the graph-cut functional to act on $L^\infty$ functions.
Second, Braides et al. worked with dense sequences of graphs that converge to $L^\infty$ graphons, whereas we proved $\Gamma$-convergence for more general sequences of graphs that converge to $L^p$ graphons. 

In the present paper, we have been concerned with theoretical ideas, but one can also examine applications of graphon GL functionals that extend existing applications of graph GL functionals. In particular, graph GL and TV functionals have been employed in image processing~\cite{caffarelli2009nonlocal, elbouchairi2023nonlocal, gilboa2009nonlocal, liu2014new, lou2010image, arias2012nonlocal, aujol2015fundamentals, el2014mean, hafiene2018nonlocal}, and it is worthwhile to pursue analogous applications of graphon GL and TV functionals. One benefit of graphon theory is that graphons provide a way to compare different networks that have distinct sizes and densities \cite{wolfe2013nonparametric}. Furthermore, because graphons are a familiar object in nonlocal analysis, one can use computational techniques from nonlocal analysis (e.g., see \cite{du2019nonlocal}) for applications that involve large graphs.

Another question for future research is how to apply our results in practical computations. For example, it is important to investigate convergence rates. For approximately what graph sizes $n$ are the minimizers of a graphon GL functional close enough to the minimizers of an associated graph GL functional? Furthermore, minimizing a graph GL functional requires approximation algorithms~\cite{budd2019mbo, merkurjev2013mbo, bertozzi2016diffuse}, and similar algorithms may be useful for minimizing a graphon GL functional. Seeking graphon GL minimizers will also involve seeking Young-measure minimizers, and this in turn will require numerical approximations~\cite{nicolaides1993computation, carstensen2000numerical}. 


\section*{Acknowledgements}

Our research was supported in part by the National Science Foundation through grants DMS-2309245 (QD), DMS-1937254 (QD, JS, and EJZ), DMS-1916439 (which is a grant to the American Mathematical Society), and DGE-2036197 (EJZ). Some of the research was {conducted} as part of EJZ's visits to UCLA. We thank Andrew Blumberg for useful discussions and two anonymous referees for helpful comments. MAP thanks Leonid Bunimovich for his mentorship when he was a postdoc at Georgia Tech and is honored by the invitation to contribute to this Festschrift.





\end{document}